\documentclass[11pt, twoside]{article}

\usepackage[english]{babel}

\usepackage{amssymb}
\usepackage{amsfonts}
\usepackage{amsmath}
\usepackage{amsthm}
\usepackage{color}
\usepackage{mathrsfs}
\usepackage{txfonts}
\usepackage{enumerate}
\usepackage{anysize}
\usepackage{indentfirst}
\usepackage{latexsym}
\usepackage{tabularx}

\usepackage[colorlinks=true,
  linkcolor=blue,
  citecolor=red,
  urlcolor=magenta,
  backref=page
  ]{hyperref}

\allowdisplaybreaks

\newtheorem{theorem}{Theorem}[section]
\newtheorem{lemma}[theorem]{Lemma}
\newtheorem{corollary}[theorem]{Corollary}
\newtheorem{proposition}[theorem]{Proposition}

\theoremstyle{definition}
\newtheorem{remark}[theorem]{Remark}
\newtheorem{definition}[theorem]{Definition}

\numberwithin{equation}{section}

\begin{document}

\arraycolsep=1pt

\title{\bf\Large
Endpoint Lorentz Inequality for Riesz Potentials
with Respect to General Measures
\footnotetext{\hspace{-0.35cm} 2020
\emph{Mathematics Subject Classification}.
Primary 42B25; Secondary 46E30, 46E35, 47B06, 26D10.
\endgraf \emph{Key words and phrases}.
Riesz potential, Lorentz space, Sobolev inequality,
maximal function, set of finite perimeter.
\endgraf
This project is partially supported by the National
Natural Science Foundation of China (Grant Nos. 12431006 and
12371093), the Beijing Natural Science Foundation (Grant No. 1262011),
and the Fundamental Research Funds for the
Central Universities (Grant No. 2253200028).}}
\author{Pingxu Hu,
Dachun Yang and Wen Yuan}
\maketitle

\vspace{-0.8cm}

\begin{center}
\begin{minipage}{13cm}
{\small {\bf Abstract}\quad
Let $n\ge2$, $\alpha\in(0,n)$, and $\mu$ be a locally
finite nonnegative Borel measure on $\mathbb{R}^n$.
We prove that, for any
$f\in C_{\rm c}^\infty(\mathbb{R}^n)$,
\begin{align*}
\|I_\alpha f\|_{L^{\frac{n}{n-\alpha},1}(\mu)}
\le C_{n,\alpha}
\int_{\mathbb{R}^n}
|Rf(x)|[M\mu(x)]^{1-\frac{\alpha}{n}}\,dx,
\end{align*}
where $I_\alpha$ is the Riesz potential,
$R$ is the Riesz transform,
$M\mu$ is the Hardy--Littlewood maximal function of $\mu$,
and $C_{n,\alpha}$ is a positive constant depending only
on $n$ and $\alpha$.
This answers a question posed by Bortz, Moen, Olivo,
P\'erez, and Rela, with the case $\alpha=1$ already known,
and extends Spector's optimal unweighted Lorentz inequality
to general measures.
As applications, we obtain a Lorentz estimate for curl-free
vector fields, Sobolev--Lorentz and Hardy inequalities
with respect to general measures, and an $A_1$-weighted
Lorentz estimate with a sharp exponent of the $A_1$
characteristic.
}
\end{minipage}
\end{center}

\tableofcontents

\section{Introduction}

Throughout the article, we assume that $n\ge2$ and that
all scalar-valued functions are real-valued.
We denote by $\mathscr{M}({\mathbb R}^n)$ the set of all
Lebesgue measurable functions on ${\mathbb R}^n$.
For any $\alpha\in(0,n)$, the \emph{Riesz
potential} $I_\alpha f$ of $f\in \mathscr{M}({\mathbb R}^n)$ is defined,
whenever the integral exists, by
\begin{align*}
I_\alpha f(x):=\frac{1}{\gamma(\alpha)}
\int_{{\mathbb R}^n}\frac{f(y)}{|x-y|^{n-\alpha}}\,dy,
\quad\forall\,
x\in \mathbb{R}^n,
\end{align*}
where $\gamma(\alpha):=
\frac{2^\alpha\pi^{{\frac{n}{2}}}
\Gamma({\frac{\alpha}{2}})}{\Gamma({\frac{n-\alpha}{2}})}$
and $\Gamma$ is the Euler
gamma function.
For any $f\in C_{\rm c}^\infty({\mathbb R}^n)$ (the set of all
infinitely differentiable functions on ${\mathbb R}^n$
with compact support), the
\emph{vector-valued Riesz transform} $Rf$ is defined by
\begin{align*}
Rf(x):=\nabla I_1f(x)
=-{\rm p.v.}\int_{{\mathbb R}^n}
\frac{c_n y}{|y|^{n+1}}f(x-y)\,dy,
\quad \forall\,x\in{\mathbb R}^n,
\end{align*}
where $\nabla u$ denotes the weak gradient of $u$ (see Definition~\ref{def:weakgradient}),
${\rm p.v.}$ denotes the principal value, and
$c_n:=\Gamma(\frac{n+1}{2})/\pi^{\frac{n+1}{2}}$.
The singular integral is understood at points where the
principal value exists.
These operators are fundamental tools in harmonic analysis
and partial differential equations; see, for example,
\cite{GrafakosMFA,LiuXiao2017,LiuXiao2021,LiuXiao2025,Ponce2016,Stein1970}.

The classical Hardy--Littlewood--Sobolev theorem
(see, for instance, \cite[Theorem 1.2.3]{GrafakosMFA}) asserts that,
when $\alpha\in(0,n)$,
$r\in (1,{\frac{n}{\alpha}})$, and ${\frac{1}{p}}={\frac{1}{r}}-{\frac{\alpha}{n}}$,
for any $f\in L^{r}({\mathbb R}^n)$,
\begin{align}\label{eq:IALPha}
\|I_\alpha f\|_{L^{p}({\mathbb R}^n)}
\le C_{n,\alpha,r}\|f\|_{L^{r}({\mathbb R}^n)};
\end{align}
moreover, O'Neil's convolution theorem \cite{ONeil1963} refines
the target space $L^{p}({\mathbb R}^n)$ to
the Lorentz space $L^{p,r}({\mathbb R}^n)$; that is,
for any $f\in L^r({\mathbb R}^n)$,
\begin{align*}
\|I_\alpha f\|_{L^{p,r}({\mathbb R}^n)}
\le C_{n,\alpha,r}\|f\|_{L^{r}({\mathbb R}^n)}.
\end{align*}
Here, $C_{n,\alpha,r}$ is a positive constant
depending only on $n$, $\alpha$, and $r$.
However, at the endpoint $r=1$, \eqref{eq:IALPha} with target space
$L^{{\frac{n}{n-\alpha}}}({\mathbb R}^n)$ fails; see \cite[p.~119]{Stein1970}.

A classical way to recover a strong endpoint estimate at the endpoint $r=1$   is to
replace $L^1({\mathbb R}^n)$ by the real Hardy space $\mathcal H^1({\mathbb R}^n)$.
Stein and Weiss
\cite[p.~31]{SteinWeiss1960} proved that, for any $f\in \mathcal H^1({\mathbb R}^n)$,
\begin{align*}
\|I_\alpha f\|_{L^{\frac{n}{n-\alpha}}({\mathbb R}^n)}
\le C_{n,\alpha}\|f\|_{\mathcal H^1({\mathbb R}^n)}
\sim
\|f\|_{L^1({\mathbb R}^n)}
+\|Rf\|_{L^1({\mathbb R}^n;{\mathbb R}^n)}.
\end{align*}
In 2007, Schikorra,
Spector, and Van Schaftingen \cite{SSVS2017} proved that
the Riesz transform term
$\|Rf\|_{L^1({\mathbb R}^n;{\mathbb R}^n)}$ alone
contains enough cancellation to control the potential term
$\|I_\alpha f\|_{L^{\frac{n}{n-\alpha}}({\mathbb R}^n)}$;
that is, for any $f\in C_{\rm c}^\infty({\mathbb R}^n)$,
\begin{align*}
\|I_\alpha f\|_{L^{{\frac{n}{n-\alpha}}}({\mathbb R}^n)}
\le C_{n,\alpha}\|Rf\|_{L^1({\mathbb R}^n; {\mathbb R}^n)}.
\end{align*}
Spector \cite{Spector2020} further established
the optimal Lorentz refinement by showing that,
for any $f\in C_{\rm c}^\infty({\mathbb R}^n)$,
\begin{align}\label{eq:Spector}
\|I_\alpha f\|_{L^{\frac{n}{n-\alpha},1}({\mathbb R}^n)}
\le C_{n,\alpha}\|Rf\|_{L^1({\mathbb R}^n;{\mathbb R}^n)}
\end{align}
and gave the related
vector-valued formulation:
for any $F\in L^{1}({\mathbb R}^n;{\mathbb R}^n)$
with $\operatorname{curl}F=0$ in the sense of distributions,
\begin{align*}
\|I_\alpha F\|_{L^{\frac{n}{n-\alpha},1}({\mathbb R}^n;{\mathbb R}^n)}
\le C_{n,\alpha}\|F\|_{L^1({\mathbb R}^n;{\mathbb R}^n)}.
\end{align*}
Here and throughout the article,
the \emph{curl-free condition $\operatorname{curl}F=0$} means that,
for any
$i,j\in\{1,\ldots,n\}$ and $\varphi\in C_{\rm c}^\infty({\mathbb R}^n)$,
\begin{align*}
\int_{{\mathbb R}^n}
\left[F_j(x)\,\partial_i\varphi(x)
-F_i(x)\,\partial_j\varphi(x)\right]\,dx=0.
\end{align*}
For further study of endpoint estimates under
differential constraints, see
\cite{BourgainBrezis2007,BCS2026,BCS2026jfaa,GHS2026arxiv,VanSchaftingen2013}.

A related direction concerns weighted estimates and trace
inequalities for general measures.
Weighted estimates for
Riesz potentials and their sharp dependence on the weight are studied
in \cite{MuckenhouptWheeden1974,LaceyMoenPerezTorres2010}.
Related results include weighted Sobolev inequalities
\cite{Perez1995} and quantitative Poincar\'e--Sobolev inequalities
\cite{PerezRela2019}.
For general measures, trace inequalities for $BV$ functions are closely
related to geometric growth conditions on the measure
\cite{MeyersZiemer1977}. Isoperimetric inequalities and the coarea
formula provide a way to prove such estimates: one first estimates
the measure of a set in terms of its perimeter, and then applies
this estimate to the level sets of   functions under consideration; see
\cite{FedererFleming1960,Federer1969,Mazya2011,Ponce2016}.
We also refer to   more recent results concerning isoperimetric estimates for general
measures \cite{PerezRela2023} and weighted fractional Poincar\'e
inequalities \cite{MyyrylainenPerezWeigt2024}.
These results motivate the use of a power of $M\mu$ as the weight
on the right-hand side of an endpoint Sobolev inequality.

For any locally finite nonnegative Borel measure $\mu$ and
for any $x\in {\mathbb R}^n$, let
\begin{align*}
M\mu(x):=\sup_{r\in (0,\infty)}\frac{\mu(B(x,r))}{|B(x,r)|}.
\end{align*}
Bortz et al.
\cite{BMOPR2026} recently
established a new global endpoint Sobolev inequality.
One of its consequences (see \cite[Corollary 2.12]{BMOPR2026})
states that, for any $u\in C_{\rm c}^\infty({\mathbb R}^n)$,
\begin{align}\label{eq:known-first-order}
\|u\|_{L^{\frac{n}{n-1},1}(\mu)}
\le C_n\int_{{\mathbb R}^n}|\nabla u(x)|
[M\mu(x)]^{\frac{n-1}{n}}\,dx.
\end{align}
By taking $u:=I_1f$ and using a cutoff and approximation argument,
\eqref{eq:known-first-order} yields the
measure analogue of \eqref{eq:Spector} at $\alpha=1$.
When $\alpha\in (1,n)$, their Theorem~2.17 establishes a
weighted $L^{{\frac{n}{n-\alpha}}}$
estimate with an Orlicz maximal operator. They then asked
in \cite[Remark~2.18]{BMOPR2026} whether
\begin{align*}
\|I_\alpha f\|_{L^{\frac{n}{n-\alpha},1}(\mu)}
\le C_{n,\alpha}
\int_{{\mathbb R}^n}|Rf(x)|[M\mu(x)]^{1-{\frac{\alpha}{n}}}\,dx
\end{align*}
holds for general locally finite measures $\mu$ when
$\alpha\neq 1$. Our main theorem below answers this question affirmatively.

\begin{theorem}\label{thm:main}
Let $\alpha\in(0,n)$, $p={\frac{n}{n-\alpha}}$, and $\mu$
be a locally finite nonnegative
Borel measure on ${\mathbb R}^n$.
Then there exists a positive constant $C_{n,\alpha}$, depending only on
$n$ and $\alpha$, such that, for any $f\in C_{\rm c}^\infty({\mathbb R}^n)$,
\begin{align}\label{eq:main}
\|I_\alpha f\|_{L^{p,1}(\mu)}
\le C_{n,\alpha}
\int_{{\mathbb R}^n}|Rf(x)|[M\mu(x)]^{\frac{1}{p}}\,dx.
\end{align}
\end{theorem}

\begin{remark}
\begin{enumerate}[{\rm (i)}]
\item If $\mu:=dx$, then $M\mu\equiv 1$.
In this case,
\eqref{eq:main} coincides with Spector's estimate
\eqref{eq:Spector} for $\alpha\in(0,n)$.

\item We note that Theorem~\ref{thm:main}
in the case $\alpha=1$ follows from \cite[Corollary 2.12]{BMOPR2026}.
However, our proof differs from the proof of \cite[Corollary 2.12]{BMOPR2026}.

\item
For general weights,
Theorem \ref{thm:main} when $\alpha\in (1,n)$ improves
\cite[Theorem~2.17]{BMOPR2026}.
Indeed,
if $\alpha\in (1,n)$, $p={\frac{n}{n-\alpha}}$, and $v$ is a weight,
then \cite[Theorem~2.17]{BMOPR2026} implies that, for any
$\epsilon\in (0,\infty)$ and $f\in C_{\rm c}^\infty({\mathbb R}^n)$,
\begin{align}\label{eq:mrem1}
\|I_\alpha f\|_{L^p(v)}
\le C_{n,\alpha,\epsilon}
\int_{{\mathbb R}^n}|Rf(x)|[M_\Psi v(x)]^{{\frac{1}{p}}}\,dx,
\end{align}
where, for any $t\in (0,\infty)$,
\begin{align*}
\Psi(t):=t[\log(e+t)]^{{\frac{1}{n-\alpha}}+\epsilon},
\end{align*}
$M_\Psi$ denotes the associated Orlicz maximal operator (see \cite[(37)]{BMOPR2026}
for its definition),
and $C_{n,\alpha,\epsilon}$ is a positive constant
depending only on $n$, $\alpha$, and $\epsilon$.
In contrast, taking $d\mu=v\,dx$ in \eqref{eq:main} gives
\begin{align*}
\|I_\alpha f\|_{L^{p,1}(v)}
\le C_{n,\alpha}
\int_{{\mathbb R}^n}|Rf(x)|[Mv(x)]^{{\frac{1}{p}}}\,dx.
\end{align*}
This improves both sides of \eqref{eq:mrem1}.
Indeed, on the left-hand side,
$L^{p,1}(v)\hookrightarrow L^p(v)$. On the right-hand side, $\Psi(t)\ge t$ implies
$Mv\lesssim M_\Psi v$. Thus,
\begin{align*}
\|I_\alpha f\|_{L^p(v)}
\lesssim \|I_\alpha f\|_{L^{p,1}(v)}
\lesssim \int_{{\mathbb R}^n}|Rf(x)|[Mv(x)]^{{\frac{1}{p}}}\,dx
\lesssim \int_{{\mathbb R}^n}|Rf(x)|[M_\Psi v(x)]^{{\frac{1}{p}}}\,dx,
\end{align*}
and hence the above assertion holds.
\end{enumerate}
\end{remark}

The arguments used to prove Theorem \ref{thm:main} also yield the following
measure-theoretic analogue of Spector's curl-free estimate
\cite[Theorem~1.1]{Spector2020}.

\begin{corollary}\label{cor:curl-free}
Let $\alpha\in(0,n)$, $p={\frac{n}{n-\alpha}}$, and $\mu$ be a locally finite
nonnegative Borel measure on ${\mathbb R}^n$.
Then there exists a positive constant $C_{n,\alpha}$, depending only
on $n$ and $\alpha$, such that,
for any $F\in L^1({\mathbb R}^n;{\mathbb R}^n)$ with $\operatorname{curl}F=0$,
\begin{align}\label{eq:curl-free-measure}
\|I_\alpha F\|_{L^{p,1}(\mu;{\mathbb R}^n)}
\le C_{n,\alpha}
\int_{{\mathbb R}^n}|F(x)|[M\mu(x)]^{\frac{1}{p}}\,dx.
\end{align}
\end{corollary}

\begin{remark}
Corollary~\ref{cor:curl-free} when $\mu:=dx$
reduces to
\cite[Theorem~1.1]{Spector2020}.
\end{remark}

For any $\alpha\in(0,1)$ and $u\in C_{\rm c}^\infty({\mathbb R}^n)$,
the \emph{Riesz fractional gradient} $\nabla^\alpha u$
is defined by setting, for any $x\in{\mathbb R}^n,$
\begin{align*}
\nabla^\alpha u(x)
:=\nabla I_{1-\alpha}u(x)
=\kappa_{n,\alpha}\int_{{\mathbb R}^n}
\frac{u(x)-u(y)}{|x-y|^{n+\alpha}}
\frac{x-y}{|x-y|}\,dy,
\end{align*}
where $\kappa_{n,\alpha}$ is a positive normalization constant
depending only on $n$ and $\alpha$; see
\cite[Definition~1.1]{ShiehSpector2015}
and \cite[pp.~246--247]{Ponce2016}.
When $\alpha=1$, we let $\nabla^1u:=\nabla u$.

We obtain the following Sobolev--Lorentz inequalities.

\begin{corollary}\label{cor:fractional-gradient}
Let $\alpha\in(0,1]$, $p={\frac{n}{n-\alpha}}$,
and $\mu$ be a locally finite nonnegative
Borel measure on ${\mathbb R}^n$.
Then there exists a positive constant $C_{n,\alpha}$, depending only
on $n$ and $\alpha$, such that, for any
$u\in C_{\rm c}^\infty({\mathbb R}^n)$,
\begin{align}\label{eq:fractional-gradient-Lorentz}
\|u\|_{L^{p,1}(\mu)}
\le C_{n,\alpha}
\int_{{\mathbb R}^n}|\nabla^\alpha u(x)|[M\mu(x)]^{\frac{1}{p}}\,dx.
\end{align}
\end{corollary}

\begin{remark}
When $\alpha\in (0,1)$ and $\mu :=dx$,
Corollary \ref{cor:fractional-gradient}
in this case reduces to
\cite[(1.3)]{Spector2020}.
Corollary~\ref{cor:fractional-gradient} in the case $\alpha=1$
reduces to \cite[Corollary~2.12]{BMOPR2026}
with $q:=n/(n-1)$.
\end{remark}

The Lorentz estimate in Corollary~\ref{cor:fractional-gradient}
also gives the following measure-theoretic analogue of
\cite[Theorem~1.4]{Spector2020}.

\begin{corollary}\label{cor:fractional-hardy-pole}
Let $\alpha\in(0,1]$ and $\mu$ be a locally finite nonnegative Borel
measure on ${\mathbb R}^n$.
Then there exists a positive constant $C_{n,\alpha}$, depending only
on $n$ and $\alpha$, such that, for any $a\in{\mathbb R}^n$ satisfying
$M\mu(a)<\infty$ and for any $u\in C_{\rm c}^\infty({\mathbb R}^n)$,
\begin{align}\label{eq:fractional-hardy-pole}
\int_{{\mathbb R}^n}\frac{|u(x)|}{|x-a|^\alpha}\,d\mu(x)
\le C_{n,\alpha}[M\mu(a)]^{\frac{\alpha}{n}}
\int_{{\mathbb R}^n}|\nabla^\alpha u(x)|
[M\mu(x)]^{1-{\frac{\alpha}{n}}}\,dx.
\end{align}
\end{corollary}

\begin{remark}
When $\alpha\in (0,1)$ and $\mu :=dx$, Corollary~\ref{cor:fractional-hardy-pole}
reduces to \cite[Theorem~1.4]{Spector2020}.
\end{remark}

The following result is a consequence of Theorem~\ref{thm:main} with
$\mu:=v\,dx$, where $v$ is a Muckenhoupt $A_1$
weight (see Definition~\ref{def:A1}).

\begin{corollary}\label{cor:A1}
Let $\alpha\in(0,n)$, $p={\frac{n}{n-\alpha}}$, and $v\in A_1$.
Then, for any $f\in C_{\rm c}^\infty({\mathbb R}^n)$,
\begin{align}\label{eq:weight}
\|I_\alpha f\|_{L^{p,1}(v)}
\le C_{n,\alpha}[v]_{A_1}^{\frac{1}{p}}
\int_{{\mathbb R}^n}|Rf(x)|v(x)^{\frac{1}{p}}\,dx.
\end{align}
Moreover, the power $\frac{1}{p}=1-{\frac{\alpha}{n}}$ of $[v]_{A_1}$ is sharp.
\end{corollary}

We now outline the proof of Theorem~\ref{thm:main}.
The proof relies on Lorentz estimates for fractional differences
and Riesz potentials of characteristic functions of sets of
finite perimeter, established in Propositions~\ref{pro:11}
and \ref{pro:22}, respectively.
The main difficulty is to control the contributions from
different scales by a single weighted perimeter integral.
To this end, we first derive pointwise dyadic estimates
and group the resulting cubes according to
$m_Q$ [see \eqref{eq:mQ}], which is defined to be
the smaller
of the densities of a set and its complement.
The relative isoperimetric inequality (see Lemma~\ref{lem:iso}) provides a summable
factor in the density parameter, while a nonlinear counting
estimate controls the sum over scales within each density layer.
Applying the resulting characteristic-function estimates
to the level sets of $I_1f$, together with the
layer-cake representation and the weighted coarea formula
(see Lemma~\ref{lem:coarea}), we then obtain the main theorem.

The remainder of this article is organized as follows.
In Section~\ref{sec:preliminaries}, we recall the definitions
and preliminary results used throughout the article.
In Section~\ref{sec:stopping}, we establish pointwise dyadic
estimates and a nonlinear counting estimate, and use them
to prove the fractional difference estimate in
Proposition~\ref{pro:11}.
In Section~\ref{sec:alpha-greater-one}, we prove the Riesz
potential estimate in Proposition~\ref{pro:22}.
Finally, Section~\ref{sec:proofs} is devoted to the proofs
of Theorem~\ref{thm:main} and its corollaries.

We conclude this introduction with some notational conventions.
Throughout the article, let $\mathbb{N}:=\{1,2,\ldots\}$.
If $E$ is a subset of ${\mathbb{R}^n}$, we denote by
${\bf 1}_E$ its characteristic function.
We use ${\bf 0}$ to denote the origin of ${\mathbb{R}^n}$ and
$\mathbb{S}^{n-1}$ to denote the unit sphere in ${\mathbb{R}^n}$.
For any $x\in {\mathbb{R}^n}$ and $r\in (0,\infty)$, define
$B(x,r):=\{y\in {\mathbb{R}^n}:|x-y|<r\}$.
For any cube $Q$ in
$\mathbb{R}^n$, $\ell (Q)$ denotes its edge length and
$Q^\circ$ its \emph{interior}.
We denote by $\mathcal H^{n-1}$ the $(n-1)$-dimensional Hausdorff measure.
For any measurable function $f$ on $\mathbb{R}^n$, its support is denoted
by $\operatorname{supp}(f)$.
We use $C_{\rm c}^{\infty}(\mathbb{R}^n)$ to denote the space of all infinitely
differentiable functions on $\mathbb{R}^n$ with compact support.
For any $k\in \mathbb{N}$ and $p\in [1,\infty)$, we use
$W^{k,p}_{\rm loc}(\mathbb{R}^n)$ to denote the set of all
$f\in L^p_{\rm loc}(\mathbb{R}^n)$ (the set of all locally $p$-integrable functions
on $\mathbb{R}^n$) such
that, for any bounded open set $U\subset \mathbb{R}^n$, $f\in W^{k,p}(U)$.
In addition, $C$ denotes a positive constant that is independent of the
main parameters involved, but may vary from line to line, whereas
$C_{\alpha,\dots}$ denotes a positive constant depending on the indicated
parameters $\alpha,\dots$.
The notation $f\lesssim g$ means $f\leq Cg$ and, if
$f\lesssim g\lesssim f$, then we write $f\sim g$.
Finally, throughout the proofs we retain the notation introduced in the
corresponding theorem or related statement.

\section{Preliminaries}\label{sec:preliminaries}

In this section,
we
present some necessary definitions and tools.
We begin with the definition of Lorentz spaces
(see \cite{GrafakosCFA}).

\begin{definition}
Let $p\in[1,\infty)$, $q\in(0,\infty]$, and $\mu$ be a
locally finite nonnegative Borel measure on ${\mathbb R}^n$.
The \emph{Lorentz space} $L^{p,q}(\mu)$ is defined to be
the set of all $\mu$-measurable functions $f$ on ${\mathbb R}^n$ such that
$\|f\|_{L^{p,q}(\mu)}<\infty$, where, for $q\in(0,\infty)$,
\begin{align}\label{eq:Lorentz-distribution}
\|f\|_{L^{p,q}(\mu)}
:=
\left[
p\int_0^\infty
t^q\mu(\{x\in{\mathbb R}^n:|f(x)|>t\})^{\frac{q}{p}}
\frac{dt}{t}
\right]^{\frac{1}{q}},
\end{align}
and, for $q=\infty$,
\begin{align}\label{eq:weak-Lorentz-distribution}
\|f\|_{L^{p,\infty}(\mu)}
:=
\sup_{t\in(0,\infty)}
t\mu(\{x\in{\mathbb R}^n:|f(x)|>t\})^{\frac{1}{p}}.
\end{align}
In particular, $L^{p,p}(\mu)=L^p(\mu)$ with equality of norms,
and $L^{p,\infty}(\mu)$ is also called the
\emph{weak Lebesgue space}.
\end{definition}

We recall the definitions of locally bounded variation and finite perimeter;
see \cite[Definition~5.2]{EvansGariepy2015}.
\begin{definition}
A function $u\in L^1_{\mathrm{loc}}({\mathbb R}^n)$ is said
to have \emph{locally bounded variation} if, for any bounded open set
$U\subset{\mathbb R}^n$,
\begin{align*}
\sup
\left\{
\int_U u(x)\operatorname{div}\Phi(x)\,dx:
\Phi\in C_{\rm c}^1(U;{\mathbb R}^n),\
\|\Phi\|_{L^\infty(U;{\mathbb R}^n)}\le1\right\}
<\infty.
\end{align*}
Let
$BV_{\mathrm{loc}}({\mathbb R}^n)$ denote the space of such functions.
Here, $C_{\rm c}^1(U;{\mathbb R}^n)$ denotes the set of all
 differentiable functions from $U$ to ${\mathbb R}^n$
with compact support in $U$.
\end{definition}

We recall the notions of distributional and weak gradients;
see \cite[Sections~4.1 and~5.1]{EvansGariepy2015}.

\begin{definition}\label{def:weakgradient}
Let $u\in L^1_{\rm loc}({\mathbb R}^n)$.
The \emph{distributional gradient} $Du$ of $u$ is the
vector-valued distribution defined by
\begin{align*}
\langle Du,\Phi\rangle
:=-\int_{{\mathbb R}^n}u(x)\operatorname{div}\Phi(x)\,dx,
\quad \forall\,\Phi\in C_{\rm c}^\infty({\mathbb R}^n;{\mathbb R}^n).
\end{align*}
Here, $C_{\rm c}^\infty({\mathbb R}^n;{\mathbb R}^n)$ denotes the set of all
infinitely differentiable functions from ${\mathbb R}^n$ to ${\mathbb R}^n$
with compact support.
If there exists $g\in L^1_{\rm loc}({\mathbb R}^n;{\mathbb R}^n)$ such that, for any
$\Phi\in C_{\rm c}^\infty({\mathbb R}^n;{\mathbb R}^n)$,
\begin{align*}
\langle Du,\Phi\rangle
=\int_{{\mathbb R}^n}g(x)\cdot\Phi(x)\,dx,
\end{align*}
then $g$ is called the \emph{weak gradient} of $u$
and is denoted by $\nabla u$.
\end{definition}

If $u\in W^{1,1}_{\rm loc}({\mathbb R}^n)$,
then $Du=\nabla u$, and we identify the distributional
gradient $Du$ with the weak gradient $\nabla u$.
By \cite[Theorem~5.1]{EvansGariepy2015}, we find that,
if $u\in BV_{\mathrm{loc}}({\mathbb R}^n)$, then its distributional gradient
$Du$ is an ${\mathbb R}^n$-valued Radon measure
with locally finite total variation.
We denote the \emph{total variation measure} of $Du$ by $|Du|$.

\begin{definition}
A Lebesgue measurable set $E\subset{\mathbb R}^n$ is said to have
\emph{locally finite perimeter} if
${\mathbf 1}_E\in BV_{\rm loc}({\mathbb R}^n)$.
For such a set and any open set $\Omega\subset{\mathbb R}^n$,
the \emph{perimeter of $E$ in $\Omega$} is defined by
\begin{align*}
P(E;\Omega):=|D{\mathbf 1}_E|(\Omega).
\end{align*}
We write $P(E):=P(E;{\mathbb R}^n)$ and say that $E$ has
\emph{finite perimeter} if $P(E)<\infty$.
\end{definition}

The reduced boundary identifies the part of the boundary that carries
the perimeter measure; see \cite[Definition~5.4]{EvansGariepy2015}.
\begin{definition}
Let $E\subset{\mathbb R}^n$ be a set of locally finite perimeter.
The \emph{reduced boundary} of $E$, denoted by $\partial^*E$,
is the set of all points $x\in{\mathbb R}^n$ satisfying the following conditions:
\begin{enumerate}[{\rm(i)}]
\item For any $r\in(0,\infty)$,
$|D{\mathbf 1}_E|(B(x,r))>0.$
\item The limit
$\nu_E(x):=
-\lim_{r\to0^+}
\frac{D{\mathbf 1}_E(B(x,r))}{|D{\mathbf 1}_E|(B(x,r))}$
exists.
\item $|\nu_E(x)|=1$.
\end{enumerate}
Here, and thereafer, the \emph{notation $r\to0^+$} means that
there exists a constant $c_0\in (0,1)$ such that $r\in(0,c_0)$ and $r\to0$.
\end{definition}

The following lemma is the relative isoperimetric inequality (see \cite[Theorem~5.11]{EvansGariepy2015} and
\cite[Lemma~2.2]{MyyrylainenPerezWeigt2024}).
\begin{lemma}\label{lem:iso}
Let $Q\subset{\mathbb R}^n$ be a cube and $E\subset{\mathbb R}^n$ be a set of finite perimeter.
Then there exists a positive constant $C_n$, depending only on
$n$, such that
\begin{align}\label{eq:reiso}
\min\{|E\cap Q|,|Q\setminus E|\}^{\frac{n-1}{n}}
\le C_n {\mathcal H}^{n-1}(\partial^*E\cap Q^\circ),
\end{align}
where $Q^\circ$ denotes the interior of $Q$.
\end{lemma}

The following weighted coarea formula can be deduced
from the measure-valued coarea formula (see, for instance, \cite[Theorem~3.40]{AFP2000}) and
the monotone convergence theorem; we omit the details.
\begin{lemma}\label{lem:coarea}
Let $u\in W^{1,1}_{\rm loc}({\mathbb R}^n)$,
and let $w$ be a nonnegative Borel measurable function
on ${\mathbb R}^n$.
Then the superlevel set
$\{u>t\}:=\{x\in{\mathbb R}^n:u(x)>t\}$
has locally finite perimeter for almost every $t\in{\mathbb R}$,
and
\begin{align*}
\int_{{\mathbb R}^n}|\nabla u(x)|w(x)\,dx
=
\int_{{\mathbb R}}\int_{\partial^*\{u>t\}}
w\,d{\mathcal H}^{n-1}\,dt.
\end{align*}
\end{lemma}

For any $\nu\in\{0,\frac{1}{3},\frac{2}{3}\}^n$,
recall that the \emph{shifted
dyadic grid} $\mathscr{D}^\nu$ is defined by setting
\begin{align*}
\mathscr{D}^\nu:=\left\{2^j\left[m+[0,1)^n+(-1)^j\nu\right]:
j\in\mathbb{Z},\ m\in\mathbb{Z}^n\right\}.
\end{align*}
The cubes are half-open; in
particular, cubes of the same edge length in a fixed grid are disjoint.
The adjacent-grid property allows ball estimates to be
represented in this finite collection of grids; see, for instance,
\cite[p.\,479]{MTT02}.

\begin{lemma}\label{lemCub}
The shifted dyadic grids $\{\mathscr{D}^\nu\}_{\nu\in\{0,\frac{1}{3},\frac{2}{3}\}^n}$
have the following properties.
\begin{enumerate}[{\rm(i)}]
\item
For any $Q,P\in\mathscr{D}^\nu$ with
$\nu\in\{0,\frac{1}{3},\frac{2}{3}\}^n$,
$Q\cap P\in\{\emptyset,Q,P\}$.
\item
For any ball $B(x,r)\subset\mathbb{R}^n$, there exist an index
$\nu\in\{0,\frac{1}{3},\frac{2}{3}\}^n$ and a cube
$Q\in{\mathscr D}^{\nu}$ such that
$
B(x,r)\subset Q$ and
$\ell(Q)\leq C_n r$,
where $C_n$ is a positive constant depending only on $n$.
\end{enumerate}
\end{lemma}

Finally, we recall the weight class used in
Corollary~\ref{cor:A1};
see \cite[Definition~7.1.1]{GrafakosCFA}.
\begin{definition}\label{def:A1}
Let $v$ be a locally integrable function on ${\mathbb R}^n$ such that
$v(x)\in (0,\infty)$ for almost every $x\in{\mathbb R}^n$.
We say that $v$ belongs to the \emph{Muckenhoupt class $A_1$} if
\begin{align*}
[v]_{A_1}
:=
\sup_{\text{cube}\, Q \subset {\mathbb R}^n}
\left[\frac{1}{|Q|}\int_{Q}v(x)\,dx\right]\left\|v^{-1}\right\|_{L^\infty (Q)}
<\infty.
\end{align*}
\end{definition}

\section{Fractional Difference Estimates for Characteristic Functions}
\label{sec:stopping}

In this section, we establish a Lorentz estimate for fractional
differences of characteristic functions. We begin by introducing the notation
needed to state the result.

For any $s\in(0,1)$, $u\in\mathscr{M}({\mathbb R}^n)$,
and $x\in{\mathbb R}^n$, define
\begin{align*}
\mathcal D^su(x):=
\int_{{\mathbb R}^n}\frac{|u(x)-u(y)|}{|x-y|^{n+s}}\,dy,
\end{align*}
where the integral is understood in the extended nonnegative sense.
This quantity was used by Spector \cite{Spector2020}
in the study of endpoint Lorentz estimates.
For convenience, we let $\mathcal D^0u:=u$.
For any Lebesgue measurable set $E\subset{\mathbb R}^n$, define
\begin{align}\label{eq:E111}
E^1
:=
\left\{
x\in{\mathbb R}^n:
\lim_{r\to0^+}
\frac{|E\cap B(x,r)|}{|B(x,r)|}=1
\right\}
\end{align}
and
\begin{align}\label{eq:E000}
E^0
:=
\left\{
x\in{\mathbb R}^n:
\lim_{r\to0^+}
\frac{|E\cap B(x,r)|}{|B(x,r)|}=0
\right\}.
\end{align}
The \emph{essential boundary} of $E$ is defined as
$\partial^eE:={\mathbb R}^n\setminus(E^0\cup E^1).$

The main result of this section is as follows.

\begin{proposition}\label{pro:11}
Let $s\in[0,1)$ and $p=\frac{n}{n-1+s}$.
Let $E\subset{\mathbb R}^n$ have finite perimeter and finite Lebesgue measure,
and let $\mu$ be a locally finite nonnegative Borel measure on ${\mathbb R}^n$.
Then
\begin{align}\label{eq:characteristic}
\|\mathcal D^s{\mathbf 1}_{E^1}\|_{L^{p,1}(\mu)}
\le C_{n,s}\int_{\partial^*E}(M\mu)^{\frac{1}{p}}\,d{\mathcal H}^{n-1},
\end{align}
where
$E^1$ is as in \eqref{eq:E111} and where
$C_{n,s}$ is a positive constant depending only on $n$ and $s$.
In particular, when $s=0$,
\begin{align}\label{eq:endpoint-iso}
\mu(E^1)^{\frac{n-1}{n}}
\le C_n\int_{\partial^*E}(M\mu)^{\frac{n-1}{n}}\,d{\mathcal H}^{n-1}.
\end{align}
Moreover, if $s\in(0,1)$ and the right-hand side of
\eqref{eq:characteristic} is finite, then $\mu(\partial^eE)=0$.
\end{proposition}

The proof of Proposition~\ref{pro:11} is divided into three steps.
In Subsection~\ref{sec3.1}, we derive pointwise estimates
using shifted dyadic grids.
In Subsection~\ref{sec:boxing}, we establish a nonlinear
counting estimate that controls contributions across scales
by a weighted boundary integral.
Finally, in Subsection~\ref{sec:main-proof}, we combine these
estimates with the relative isoperimetric inequality
to prove Proposition~\ref{pro:11}.

\subsection{Pointwise Dyadic Estimates}
\label{sec3.1}

\begin{lemma}\label{lem:dyadic-reduction}
Let $s\in (0,1)$ and $E\subset{\mathbb R}^n$ be
Lebesgue measurable. Then
there exists a positive constant $C_{n,s}$,
depending only on $n$ and $s$, such that, for any $x\in{\mathbb R}^n$,
\begin{align}\label{eq:dyadic-reduction}
\mathcal D^s{\mathbf 1}_E(x)
\le C_{n,s}
\sum_{\nu\in\{0,\frac{1}{3},\frac{2}{3}\}^n}
\sum_{Q\in{\mathscr D}^\nu}
\ell(Q)^{-s}
\left|{\mathbf 1}_E(x)-e_Q \right|
{\mathbf 1}_Q(x),
\end{align}
where
\begin{align}\label{eq:density}
e_Q:=\frac{|E\cap Q|}{|Q|}.
\end{align}
\end{lemma}

\begin{proof}
Fix $x\in{\mathbb R}^n$. For any $k\in\mathbb Z$, let
\begin{align*}
A_k(x):=\{y\in \mathbb{R}^n:2^k\le|x-y|<2^{k+1}\}.
\end{align*}
By Lemma \ref{lemCub}(ii), we
find that there exists a cube
$Q_k\in \bigcup_{\nu\in\{0,\frac{1}{3},\frac{2}{3}\}^n}{\mathscr D}^{\nu}$ such that
$B(x,2^{k+1})\subset Q_k$ and
$\ell(Q_k)\sim 2^k$.
From this, it follows that
\begin{align}\label{eq:ak}
\int_{A_k(x)}
\frac{|{\mathbf 1}_E(x)-{\mathbf 1}_E(y)|}{|x-y|^{n+s}}\,dy
\le 2^{-k(n+s)}\int_{Q_k}|{\mathbf 1}_E(x)-{\mathbf 1}_E(y)|\,dy.
\end{align}
Note that ${\mathbf 1}_E(x)\in\{0,1\}$.
Thus,
if ${\mathbf 1}_E(x)=1$, then
\begin{align*}
\int_{Q_k}
|{\mathbf 1}_E(x)-{\mathbf 1}_E(y)|\,dy
=
\int_{Q_k}[1-{\mathbf 1}_E(y)]\,dy
=
|Q_k|
\left(
1-e_{Q_k}
\right);
\end{align*}
if ${\mathbf 1}_E(x)=0$, then
\begin{align*}
\int_{Q_k}
|{\mathbf 1}_E(x)-{\mathbf 1}_E(y)|\,dy
=
\int_{Q_k}{\mathbf 1}_E(y)\,dy
=
|Q_k|e_{Q_k}.
\end{align*}
Using these, we obtain
\begin{align*}
\int_{Q_k}
|{\mathbf 1}_E(x)-{\mathbf 1}_E(y)|\,dy
=
|Q_k|
\left|
{\mathbf 1}_E(x)-e_{Q_k}
\right|,
\end{align*}
which, together with \eqref{eq:ak} and $\ell(Q_k)\sim 2^k$, implies that
\begin{align*}
\int_{A_k(x)}
\frac{|{\mathbf 1}_E(x)-{\mathbf 1}_E(y)|}{|x-y|^{n+s}}\,dy
\lesssim \ell(Q_k)^{-s}
\left|{\mathbf 1}_E(x)-e_{Q_k}\right|.
\end{align*}
Since $\ell(Q_k)\sim2^k$, each selected cube occurs only a bounded
number of times. Summing in $k$, we obtain \eqref{eq:dyadic-reduction},
which completes the proof of Lemma \ref{lem:dyadic-reduction}.
\end{proof}

For any Lebesgue measurable set $E\subset{\mathbb R}^n$,
$\nu\in\{0,\frac{1}{3},\frac{2}{3}\}^n$, and $s\in[0,1)$, let $e_Q$ be as in
\eqref{eq:density} and define, for any $j\in{\mathbb N}$ and
$Q\in\mathcal {\mathscr D}^{\nu}$,
\begin{align}\label{eq:mQ}
m_Q:=\min\{e_Q,1-e_Q\},
\end{align}
\begin{align}\label{eq:Fj}
\mathcal G^{\nu}_j:=
\left\{Q\in{\mathscr D}^{\nu}:2^{-j-1}<m_Q\le2^{-j}\right\},
\end{align}
and
\begin{align}\label{eq:Gj-definition}
G^{\nu}_j(x):=\sup_{Q\in\mathcal G^{\nu}_j}
2^{-j}\ell(Q)^{-s}{\mathbf 1}_Q(x),\qquad \forall\,x\in{\mathbb R}^n,
\end{align}
where $e_Q$ is as in \eqref{eq:density}.
The supremum over an empty family is understood to be zero.

\begin{lemma}\label{lem:last-crossing}
Let $s\in (0,1)$,
$E\subset{\mathbb R}^n$ be
Lebesgue measurable, and $\nu\in\{0,\frac{1}{3},\frac{2}{3}\}^n$.
Then there exists a positive constant $C_{n,s}$,
depending only on $n$ and $s$, such that, for any
$x\in E^0\cup E^1$,
\begin{align}\label{eq:last-crossing}
\sum_{Q\in{\mathscr D}^\nu}
\ell(Q)^{-s}
|{\mathbf 1}_{E^1}(x)-e_Q|{\mathbf 1}_{Q}(x)
\le C_{n,s}\sum_{j\in {\mathbb N}} G^{\nu}_j(x),
\end{align}
where $\mathcal G^{\nu}_j$ and $G^{\nu}_j$ are as, respectively,
 in \eqref{eq:Fj} and
\eqref{eq:Gj-definition},
and where $E^0$ and $E^1$ are as, respectively,
 in \eqref{eq:E000} and \eqref{eq:E111}.
\end{lemma}

\begin{proof}
Fix $x\in E^0\cup E^1$.
Then
${\mathbf 1}_{E^1}(x)\in\{0,1\}$.
Define
\begin{align*}
\mathcal B(x):=
\left\{Q\in {\mathscr D}^\nu : {\mathbf 1}_{E^1}(x)=1\text{ and }e_Q\ge\frac{1}{2},
\text{ or }{\mathbf 1}_{E^1}(x)=0\text{ and }e_Q\le\frac{1}{2}
\right\}
\end{align*}
and
\begin{align*}
\mathcal C(x):=
\left\{Q\in {\mathscr D}^\nu : {\mathbf 1}_{E^1}(x)=1\text{ and }e_Q<\frac{1}{2},
\text{ or }{\mathbf 1}_{E^1}(x)=0\text{ and }e_Q>\frac{1}{2}
\right\}.
\end{align*}
From these definitions, we deduce that,
if $Q\in \mathcal B(x)$, then
\begin{align}\label{eq:majority-oscillation}
|{\mathbf 1}_{E^1}(x)-e_Q|=m_Q\le \frac{1}{2};
\end{align}
if $Q\in \mathcal C(x)$, then
\begin{align}\label{eq:minority-oscillation}
|{\mathbf 1}_{E^1}(x)-e_Q|=1-m_Q\le1.
\end{align}
For any $k\in\mathbb Z$, let $Q_k(x)$ denote the unique cube
in ${\mathscr D}^\nu$ of edge length $\ell(Q_k(x))=2^{-k}$ that contains $x$.
For any fixed $j\in {\mathbb N}$, using $\mathcal G^{\nu}_j\subset{\mathscr D}^\nu$
and Lemma \ref{lemCub}(i), we find that
the collection
$\{Q\in\mathcal G^{\nu}_j:x\in Q\}$
is totally ordered by inclusion and hence is a subcollection
of the dyadic chain $\{Q_k(x)\}_{k\in\mathbb Z}$.
By the definitions of $E^0$ and $E^1$ and by the inclusion
$Q_k(x)\subset B(x,\sqrt n\,2^{-k})$, we find that
\begin{align}\label{eq:dyadic-density-convergence}
e_{Q_k(x)}\to {\mathbf 1}_{E^1}(x)
\end{align}
as $k\to\infty$,
which implies that
\begin{align}\label{eq:min-density-convergence}
m_{Q_k(x)}
=
\min \left\{e_{Q_k(x)},1-e_{Q_k(x)}\right\}
\to
\min \left\{{\mathbf 1}_{E^1}(x),1-{\mathbf 1}_{E^1}(x)\right\}
=0
\end{align}
as $k\to\infty$.

We first estimate the contribution of $\mathcal B(x)$.
For any $j\in {\mathbb N}$, define
\begin{align*}
\mathcal A_j(x):
=
\mathcal B(x) \cap \{Q\in\mathcal G^{\nu}_j:x\in Q\}.
\end{align*}
If $\mathcal A_j(x)=\emptyset$, then its contribution is zero.
Therefore, suppose that $\mathcal A_j(x)\neq\emptyset$, and let
\begin{align*}
I_j(x)
:=
\left\{
k\in\mathbb Z:
Q_k(x)\in\mathcal A_j(x)
\right\}.
\end{align*}
Then $I_j(x)$ is nonempty.

We claim that $I_j(x)$ is bounded above. Indeed, by
\eqref{eq:min-density-convergence}, there exists
$K_j\in {\mathbb Z}$ such that, for any $k\ge K_j$,
$m_{Q_k(x)}\le2^{-j-1}$, which, together with the definition
of $\mathcal G^\nu_j$, implies $Q_k(x)\notin \mathcal G^\nu_j$.
Consequently,
$I_j(x)\subset(-\infty,K_j-1]\cap\mathbb Z.$
Thus, we can define
$k_j:=\max I_j(x)$ and $P_j:=Q_{k_j}(x).$
Then $P_j$ is the cube in $\mathcal A_j(x)$ having the smallest
edge length.
Using this, \eqref{eq:majority-oscillation},
and the definitions of $\mathcal A_j(x)$ and $G^{\nu}_j$, we obtain
\begin{align*}
\sum_{Q\in\mathcal A_j(x)}
\ell(Q)^{-s}
|{\mathbf 1}_{E^1}(x)-e_Q|
&=
\sum_{Q\in\mathcal A_j(x)}
\ell(Q)^{-s}m_Q
=
\sum_{k\in I_j(x)}
2^{ks}m_{Q_k(x)}
\\
&\le
2^{-j}\sum_{k\in I_j(x)}2^{ks}
\le
2^{-j}\sum_{k=-\infty}^{k_j}2^{ks}
\\
&=
2^{-j}2^{k_js}
\sum_{r=0}^{\infty}2^{-rs}
=
\frac{2^{-j}\ell(P_j)^{-s}}{1-2^{-s}}
\le
\frac{1}{1-2^{-s}}G^{\nu}_j(x).
\end{align*}
This implies
\begin{align}
\sum_{Q\in \mathcal B(x)}
\ell(Q)^{-s}|{\mathbf 1}_{E^1}(x)-e_Q|{\mathbf 1}_Q (x)
\le
\frac{1}{1-2^{-s}}
\sum_{j\in {\mathbb N}} G^{\nu}_j(x).
\label{eq:total-majority}
\end{align}

It remains to estimate the contribution of $\mathcal C(x)$. Define
\begin{align*}
I_{\rm min}(x)
:=
\left\{
k\in\mathbb Z:
Q_k(x) \in \mathcal C(x)
\right\}.
\end{align*}
If $I_{\rm min}(x)=\emptyset$, then this contribution is zero
and \eqref{eq:last-crossing} follows from \eqref{eq:total-majority}.
Suppose therefore that $I_{\rm min}(x)\ne\emptyset$.
We claim that $I_{\rm min}(x)$ is bounded above. To see this,
suppose first that ${\mathbf 1}_{E^1}(x)=1$. From this and
\eqref{eq:dyadic-density-convergence}, we infer that
$e_{Q_k(x)}\to1$ as $k\to \infty$.
Thus, there exists $K\in\mathbb Z$ such that,
for any $k\ge K$,
$e_{Q_k(x)}>\frac{1}{2}$ and hence $Q_k(x)\in \mathcal{B}(x)$.
If ${\mathbf 1}_{E^1}(x)=0$, then
$e_{Q_k(x)}\to0$
as $k\to \infty$.
Thus, there exists $K\in\mathbb Z$ such that,
for any $k\ge K$,
$e_{Q_k(x)}<\frac{1}{2}$,
and hence $Q_k(x)\in \mathcal{B}(x)$.
By these, we conclude
that $I_{\rm min}(x)\subset(-\infty,K-1]\cap\mathbb Z.$
We then define $k_*:=\max I_{\rm min}(x)$ and
$R_*:=Q_{k_*}(x).$
Then $R_*$ is the smallest cube in $\mathcal C(x)$ containing $x$.
From this and \eqref{eq:minority-oscillation}, we
deduce that
\begin{align}
\sum_{Q\in \mathcal C(x)}
\ell(Q)^{-s}|{\mathbf 1}_{E^1}(x)-e_Q|{\mathbf 1}_{Q}(x)
&\le
\sum_{Q\in \mathcal C(x)}
\ell(Q)^{-s}{\mathbf 1}_{Q}(x)
=
\sum_{k\in I_{\rm min}(x)}2^{ks}\notag\\
&\le
\sum_{k=-\infty}^{k_*}2^{ks}
=
2^{k_*s}\sum_{r=0}^{\infty}2^{-rs}
=
\frac{\ell(R_*)^{-s}}{1-2^{-s}}.
\label{eq:total-minority-first}
\end{align}

Let
$ R_*':=Q_{k_*+1}(x)$
be the dyadic child of $R_*$ containing $x$. Then $R_*'\in \mathcal B (x)$.
We next show that
\begin{align}\label{eq:last-minority-density}
m_{R_*}\ge2^{-n-1}.
\end{align}
Suppose first that ${\mathbf 1}_{E^1}(x)=1$. Using
$R_*\in \mathcal C (x)$ and $R_*'\in \mathcal B (x)$, we obtain
$e_{R_*}<\frac{1}{2}$ and $e_{R_*'}\ge\frac{1}{2}.$
From this and $|R_*'|=2^{-n}|R_*|$,
it follows that
\begin{align*}
m_{R_*}
=e_{R_*}
=
\frac{|E\cap R_*|}{|R_*|}
\ge
\frac{|E\cap R_*'|}{|R_*|}
=
\frac{|R_*'|}{|R_*|}
\frac{|E\cap R_*'|}{|R_*'|}=
2^{-n}e_{R_*'}
\ge2^{-n-1}.
\end{align*}
Suppose now that ${\mathbf 1}_{E^1}(x)=0$. By
$R_*\in \mathcal C (x)$ and $R_*'\in \mathcal B (x)$, we have
$e_{R_*}>\frac{1}{2}$ and $e_{R_*'}\le\frac{1}{2}.$
From this, we infer that
\begin{align*}
m_{R_*}
=1-e_{R_*}
=
\frac{|E^c\cap R_*|}{|R_*|}
\ge
\frac{|E^c\cap R_*'|}{|R_*|}
=
\frac{|R_*'|}{|R_*|}
\frac{|E^c\cap R_*'|}{|R_*'|}
=
2^{-n}(1-e_{R_*'})
\ge2^{-n-1}.
\end{align*}
This proves \eqref{eq:last-minority-density} in both cases.

By \eqref{eq:last-minority-density} and $m_{R_*}\le \frac{1}{2}$, we find that
$2^{-n-1}\le m_{R_*}\le\frac{1}{2}.$
Observe that
\begin{align*}
\bigcup_{j=1}^{n+1}
\left(2^{-j-1},2^{-j}\right]
=
\left(2^{-n-2},\frac{1}{2}\right].
\end{align*}
Thus, there exists a unique
$j_*\in\{1,\ldots,n+1\}$
such that $2^{-j_*-1}<m_{R_*}\le2^{-j_*},$
which implies that
$R_*\in\mathcal G^{\nu}_{j_*}.$
From this
and the
definition of $G^{\nu}_{j_*}$, we deduce that
\begin{align*}
\ell(R_*)^{-s}
\le
2^{j_*}G^{\nu}_{j_*}(x)
\le
2^{n+1}G^{\nu}_{j_*}(x)
\le
2^{n+1}\sum_{j=1}^{n+1}G^{\nu}_j(x).
\end{align*}
Combining this with \eqref{eq:total-minority-first}, we obtain
\begin{align*}
\sum_{Q\in \mathcal C(x)}
\ell(Q)^{-s}|{\mathbf 1}_{E^1}(x)-e_Q|{\mathbf 1}_{Q}(x)
\le
\frac{2^{n+1}}{1-2^{-s}}
\sum_{j=1}^{n+1}G^{\nu}_j(x),
\end{align*}
which, together with \eqref{eq:total-majority} and
${\mathscr D}^\nu= \mathcal B(x)\cup\mathcal C(x)$,
further implies \eqref{eq:last-crossing}.
This completes the proof of Lemma~\ref{lem:last-crossing}.
\end{proof}

\begin{lemma}\label{lem:zero-order-crossing}
Let $E\subset{\mathbb R}^n$ be a Lebesgue measurable set with $|E|<\infty$,
$\nu\in\{0,\frac{1}{3},\frac{2}{3}\}^n$,
and $\mathcal G^{\nu}_j$ and $G^{\nu}_j$ be as, respectively, in \eqref{eq:Fj} and
\eqref{eq:Gj-definition} with $s=0$.
Then, for any $x\in{\mathbb R}^n$,
\begin{align}\label{eq:zero-order-crossing}
{\mathbf 1}_{E^1}(x)\le 2^{n+1}\sum_{j=1}^{n+1}G^{\nu}_j(x).
\end{align}
\end{lemma}

\begin{proof}
If $x\notin E^1$, then \eqref{eq:zero-order-crossing} is immediate.
Fix $x\in E^1$. For any $k\in{\mathbb Z}$, let $Q_k(x)$ be the unique
cube in ${\mathscr D}^\nu$ of edge length $2^{-k}$ containing $x$.
By the definition of $E^1$ and
$Q_k(x)\subset B(x,\sqrt n\,2^{-k})$, we have
$e_{Q_k(x)}\to 1$ as $k\to\infty.$
On the other hand, $|E|<\infty$ implies that
\begin{align*}
0\le e_{Q_k(x)}\le |E|2^{kn}\to 0
\qquad\text{as }k\to-\infty.
\end{align*}
Thus, the set $\{k\in{\mathbb Z}:e_{Q_k(x)}<\frac{1}{2}\}$ is nonempty
and bounded above. Define
\begin{align*}
k_*:=\max\left\{k\in{\mathbb Z}:e_{Q_k(x)}<\frac{1}{2}\right\},
\quad R:=Q_{k_*}(x),\quad\text{and}\quad R':=Q_{k_*+1}(x).
\end{align*}
Then $e_R<\frac{1}{2}$ and $e_{R'}\ge\frac{1}{2}$.
Using $R'\subset R$ and $|R'|=2^{-n}|R|$, we obtain
\begin{align*}
2^{-n-1}\le 2^{-n}e_{R'}
=\frac{|E\cap R'|}{|R|}
\le\frac{|E\cap R|}{|R|}=m_R<\frac{1}{2}.
\end{align*}
Consequently, there exists $j_*\in\{1,\ldots,n+1\}$ such that
$R\in\mathcal G^\nu_{j_*}$. From this, $x\in R$, and
\eqref{eq:Gj-definition} with $s=0$, it follows that
\begin{align*}
{\mathbf 1}_{E^1}(x)=1
\le 2^{n+1}2^{-j_*}
\le 2^{n+1}G^{\nu}_{j_*}(x)
\le 2^{n+1}\sum_{j=1}^{n+1}G^{\nu}_j(x).
\end{align*}
This proves \eqref{eq:zero-order-crossing} and completes the proof
of Lemma~\ref{lem:zero-order-crossing}.
\end{proof}

\subsection{The Nonlinear Boxing Estimate}\label{sec:boxing}

\begin{lemma}\label{lem:HH}
Let $p\in(1,\infty)$
and $\mu$ be a locally finite nonnegative Borel measure.
Let $\mathcal F$ be a subcollection of a dyadic system ${\mathscr D}$
such that the cubes in $\mathcal F$ have pairwise
disjoint interiors.
Assume that $\{h_Q\}_{Q\in \mathcal F}$
and $\{w_Q\}_{Q\in \mathcal F}$ are positive sequences
such that, for any $Q\in\mathcal F$,
\begin{align}\label{eq:MUQ}
\mu (Q)^{\frac{1}{p}} \le C_0 h_Q w_Q,
\end{align}
where $C_0$ is a positive constant independent of $Q$.
Then
\begin{align}\label{eq:LamQ}
\left\|\sum_{Q\in\mathcal F}{\mathbf 1}_Q\right\|_{L^{p,1}(\mu)}
\le p C_0
\left[\sum_{Q\in\mathcal F}(h_Q w_Q)^p\right]^{\frac{1}{p}}.
\end{align}
\end{lemma}

\begin{proof}
For any $t\in (0,\infty)$, let
\begin{align*}
E_t :=\left\{x\in{\mathbb R}^n:\sum_{Q\in\mathcal F} {\mathbf 1}_Q (x)>t\right\}.
\end{align*}
From this and the disjointness of $\mathcal F$, we infer that
\begin{align}\label{eq:Gj-level-set}
E_t
=
\begin{cases}
\displaystyle\bigcup_{Q\in\mathcal F}Q,\quad &t\in (0,1),\\
\emptyset,\quad &t\in [1,\infty).
\end{cases}
\end{align}
Applying this and
the fact that every dyadic cube is a Borel set and that $\mathcal F$ is
countable, we conclude that $\sum_{Q\in\mathcal F} {\mathbf 1}_Q$
is Borel measurable and,
in particular, $\mu$-measurable.
Let
\begin{align*}
H:= \left[\sum_{Q\in\mathcal F}(h_Q w_Q)^p\right]^{\frac{1}{p}}.
\end{align*}
From \eqref{eq:Gj-level-set}, the countable subadditivity of $\mu$, and
\eqref{eq:MUQ},
we deduce that, for any $t\in (0,1)$,
\begin{align*}
\mu(E_t)
\le
\sum_{Q\in\mathcal F}
\mu(Q)
\le C_0^p
\sum_{Q\in\mathcal F}
\left(h_Q w_Q\right)^p
=C_0^p
H^p.
\end{align*}
By this and \eqref{eq:Lorentz-distribution},
we obtain
\begin{align*}
\left\|\sum_{Q\in\mathcal F} {\mathbf 1}_Q\right\|_{L^{p,1}(\mu)}
&=p\int_0^\infty
\mu(E_t)^{\frac{1}{p}}\,dt
\le p C_0 H
\int_0^1
\,dt
=
p C_0  H,
\end{align*}
which implies \eqref{eq:LamQ}.
This completes the proof of
Lemma \ref{lem:HH}.
\end{proof}

\begin{lemma}
\label{lem:supercritical-counting}
Let $p\in (1,\infty)$, $E\subset{\mathbb R}^n$ be a set of finite perimeter, and
$w$ be a nonnegative Borel measurable function.
Let $\mathcal F$ be a subcollection of a dyadic system ${\mathscr D}$.
Suppose that,
for any $Q\in\mathcal F$,
\begin{align}\label{eq:HQQ}
h_Q:=a\ell(Q)^{n-1}
\le C_n{\mathcal H}^{n-1}(\partial^*E\cap Q^\circ),
\end{align}
where $a\in (0,\infty)$ is independent of $Q$ and $C_n$ is a positive constant
depending only on $n$.
Define
\begin{align*}
w_Q:=\mathop{\operatorname{ess\,inf}}_{\partial^*E\cap Q^\circ}w
\quad \text{and}\quad
\mathcal F_k:=\{Q\in\mathcal F:\ell(Q)=2^{-k}\},
\end{align*}
where the essential infimum is taken with respect to
${\mathcal H}^{n-1}$ on $\partial^*E\cap Q^\circ$.
Then there exists a positive constant $C_{n,p}$, depending only on $n$ and $p$
and independent of $a$, such that
\begin{align}\label{eq:supercritical-counting}
\sum_{k\in\mathbb Z}
\left[\sum_{Q\in\mathcal F_k}(h_Q w_Q)^p\right]^{\frac{1}{p}}
\le C_{n,p}
\int_{\partial^*E}w\,d{\mathcal H}^{n-1}.
\end{align}
\end{lemma}

\begin{proof}
We may
assume that
\begin{align}\label{eq:finite-boundary-charge}
\int_{\partial^*E}w\,d{\mathcal H}^{n-1}<\infty.
\end{align}
For any $k\in\mathbb Z$,
if $Q\in\mathcal F_{k}$, then $\ell(Q)=2^{-k}$, and hence
\begin{align}\label{eq:hQhjk}
h_Q
=
a\ell(Q)^{n-1}
=
a2^{-k(n-1)}
=:
h_{k}.
\end{align}
For any $\tau\in (0,\infty)$, define
\begin{align*}
H(\tau):
=
{\mathcal H}^{n-1}\left(
\{z\in\partial^*E:w(z)>\tau\}
\right)\quad\mathrm{and}\quad
N_{k}(\tau):
=
\# \left\{
Q\in\mathcal F_{k}:w_Q>\tau
\right\},
\end{align*}
where $\# E$ denotes the cardinality of a set $E$.

We first show that, for any $\tau\in (0,\infty)$,
\begin{align}\label{eq:counting}
N_{k}(\tau)h_{k}
\le C_n H(\tau),
\end{align}
where $C_n$ is a positive constant depending only on $n$.
To this end, for any $\tau\in (0,\infty)$, define
\begin{align*}
\mathcal Q_{k}(\tau):
=
\left\{
Q\in\mathcal F_{k}:w_Q>\tau
\right\}.
\end{align*}
By the definition of
$w_Q$, we find that, for any $\tau\in (0,\infty)$,
the inequality $w_Q>\tau$ implies that,
for ${\mathcal H}^{n-1}$-almost every $z\in \partial^*E\cap Q^\circ$,
$w(z)>\tau$.
Therefore, for any $Q\in\mathcal Q_{k}(\tau)$,
\begin{align}\label{eq:SQ-subset}
\partial^*E\cap Q^\circ
\subset
\{z\in\partial^*E:w(z)>\tau\}
\end{align}
up to a ${\mathcal H}^{n-1}$-null set.
From this, \eqref{eq:hQhjk}, the definitions of $\mathcal Q_{k}(\tau)$ and $H(\tau)$,
\eqref{eq:HQQ},
the disjointness of the sets $\{\partial^*E\cap Q^\circ:Q\in\mathcal F_{k}\}$, and
\eqref{eq:SQ-subset},
it follows that
\begin{align*}
N_{k}(\tau)h_{k}
&=
\sum_{Q\in\mathcal Q_{k}(\tau)}h_Q
\le C_n
\sum_{Q\in\mathcal Q_{k}(\tau)}
{\mathcal H}^{n-1}(\partial^*E\cap Q^\circ)\\
&=
C_n{\mathcal H}^{n-1}\left(
\bigcup_{Q\in\mathcal Q_{k}(\tau)}\left[\partial^*E\cap Q^\circ\right]
\right)\\
&\le C_n
{\mathcal H}^{n-1}\left(
\{z\in\partial^*E:w(z)>\tau\}
\right)=C_n
H(\tau).
\end{align*}
This proves \eqref{eq:counting}.

By Chebyshev's inequality and
\eqref{eq:finite-boundary-charge}, we find that,
for any $\tau\in (0,\infty)$,
\begin{align*}
H(\tau)
\le
\frac{1}{\tau}
\int_{\partial^*E}w\,d{\mathcal H}^{n-1}
<\infty.
\end{align*}
Using this and \eqref{eq:counting}, we conclude that
$N_{k}(\tau)$ is finite.
In particular, if $N_{k}(\tau)>0$, then
$N_{k}(\tau)\ge1$, and hence \eqref{eq:counting} yields
\begin{align*}
h_{k}
\le
N_{k}(\tau)h_{k}
\le C_n H(\tau);
\end{align*}
that is,
\begin{align}\label{eq:Njq0}
N_{k}(\tau)>0
\quad\Rightarrow \quad
h_{k}\le C_nH(\tau).
\end{align}

For any $Q\in\mathcal F$, the layer-cake identity gives
\begin{align*}
w_Q
=
\int_0^\infty
{\mathbf 1}_{\{\tau<w_Q\}}\,d\tau.
\end{align*}
Applying this, $p>1$, and Minkowski's integral inequality,
we obtain
\begin{align}\label{eq:Minkowski-count}
\left(
\sum_{Q\in\mathcal F_{k}}w_Q^p
\right)^{\frac{1}{p}}
&=
\left[
\sum_{Q\in\mathcal F_{k}}
\left(
\int_0^\infty
{\mathbf 1}_{\{\tau<w_Q\}}\,d\tau
\right)^p
\right]^{\frac{1}{p}}\notag \\
&\le
\int_0^\infty
\left(
\sum_{Q\in\mathcal F_{k}}
{\mathbf 1}_{\{\tau<w_Q\}}^p
\right)^{\frac{1}{p}}
d\tau
= \int_0^\infty
N_{k}(\tau)^{\frac{1}{p}}\,d\tau.
\end{align}
When $N_{k}(\tau)>0$, from \eqref{eq:counting}, we infer that
\begin{align*}
h_{k}N_{k}(\tau)^{\frac{1}{p}}
\lesssim
h_{k}
\left[
\frac{H(\tau)}{h_{k}}
\right]^{\frac{1}{p}}
=
h_{k}^{\frac{p-1}{p}}H(\tau)^{\frac{1}{p}}.
\end{align*}
When $N_{k}(\tau)=0$, the left-hand side is zero. Applying
these and \eqref{eq:Njq0}, we conclude that
\begin{align*}
h_{k}N_{k}(\tau)^{\frac{1}{p}}
\lesssim
h_{k}^{\frac{p-1}{p}}H(\tau)^{\frac{1}{p}}
{\mathbf 1}_{\{h_{k}\le C_nH(\tau)\}}.
\end{align*}
By this and \eqref{eq:Minkowski-count}, we conclude that
\begin{align}\label{eq:boxing-intermediate}
\left[\sum_{Q\in\mathcal F_k}(h_Q w_Q)^p\right]^{\frac{1}{p}}
\lesssim
\int_0^\infty
h_{k}^{\frac{p-1}{p}}H(\tau)^{\frac{1}{p}}
{\mathbf 1}_{\{h_{k}\le C_nH(\tau)\}}
\,d\tau.
\end{align}
It remains to sum this estimate over $k$. For this purpose, we
show that, for any fixed $\tau\in (0,\infty)$,
\begin{align}\label{eq:geometric-h}
\sum_{\genfrac{}{}{0pt}{}{k\in\mathbb Z}{h_{k}\le C_nH(\tau)}}
h_{k}^{\frac{p-1}{p}}
\lesssim H(\tau)^{\frac{p-1}{p}}
\end{align}
with the implicit constant depending only on $n$ and $p$.
Fix $\tau\in (0,\infty)$. If $H(\tau)=0$, then the set
\begin{align*}
\left\{k\in\mathbb Z:h_{k}\le C_nH(\tau)\right\}
\end{align*}
is empty because $h_{k}>0$ for any $k\in \mathbb{Z}$. In this case,
\eqref{eq:geometric-h} is immediate.
Suppose now that $H(\tau)>0$.
Using $h_{k}=a2^{-k(n-1)}$, we obtain
\begin{align}\label{eq:geomet}
h_{k+1}^{\frac{p-1}{p}}
=
2^{-\frac{(n-1)(p-1)}{p}}h_{k}^{\frac{p-1}{p}}.
\end{align}
Observe that $h_{k}$ is decreasing in $k$ and that
$
h_{k}\to 0
$
as $k\to\infty$
and
$
h_{k}\to \infty
$
as $k\to-\infty$.
Therefore,
there exists a smallest integer $k_0$ such that
$h_{k_0}\le C_nH(\tau)$
and hence
\begin{align*}
\{k\in\mathbb Z:h_{k}\le C_nH(\tau)\}
=
\{k\in\mathbb Z:k\ge k_0\}.
\end{align*}
From these, \eqref{eq:geomet}, and $p\in (1,\infty)$, we deduce that
\begin{align*}
\sum_{\genfrac{}{}{0pt}{}{k\in\mathbb Z}{
h_{k}\le C_nH(\tau)}}
h_{k}^{\frac{p-1}{p}}
= \sum_{k=k_0}^{\infty}h_{k}^{\frac{p-1}{p}}
=h_{k_0}^{\frac{p-1}{p}}
\sum_{r=0}^{\infty}2^{-r\frac{(n-1)(p-1)}{p}}
\lesssim
h_{k_0}^{\frac{p-1}{p}}
\lesssim
H(\tau)^{\frac{p-1}{p}}.
\end{align*}
This proves \eqref{eq:geometric-h}.

Combining \eqref{eq:boxing-intermediate}, \eqref{eq:geometric-h}, and
Tonelli's theorem, we obtain
\begin{align}\label{eq:BjkH}
\sum_{k\in\mathbb Z}\left[\sum_{Q\in\mathcal F_k}(h_Q w_Q)^p\right]^{\frac{1}{p}}
\lesssim
\int_0^\infty H(\tau)\,d\tau.
\end{align}
Finally, by the layer-cake formula on
$(\partial^*E,{\mathcal H}^{n-1})$ and Tonelli's theorem,
we have
\begin{align*}
\int_0^\infty H(\tau)\,d\tau
&=
\int_0^\infty
{\mathcal H}^{n-1}\left(
\{z\in\partial^*E:w(z)>\tau\}
\right)\,d\tau\\
&=
\int_0^\infty
\int_{\partial^*E}
{\mathbf 1}_{\{\tau<w(z)\}}
\,d{\mathcal H}^{n-1}(z)\,d\tau\\
&=
\int_{\partial^*E}
\int_0^\infty
{\mathbf 1}_{\{\tau<w(z)\}}
\,d\tau\,d{\mathcal H}^{n-1}(z)
=
\int_{\partial^*E}
w(z)\,d{\mathcal H}^{n-1}(z),
\end{align*}
which, together with \eqref{eq:BjkH}, implies \eqref{eq:supercritical-counting}.
This completes the proof of Lemma~\ref{lem:supercritical-counting}.
\end{proof}

\subsection{Proof of Proposition~\ref{pro:11}}\label{sec:main-proof}

We now connect the geometric counting lemma with the target measure.
The following local estimate combines relative isoperimetry with the
definition of $M\mu$ and produces the factor $2^{-{\frac{j}{n}}}$. We state it for
$s\in (1-n,1)$ so that the same argument also applies to the positive
potential in Section~\ref{sec:alpha-greater-one}.

\begin{lemma}\label{lem:local-charge}
Let $E\subset {\mathbb R}^n$ be of finite perimeter
and $\mu$ be a locally finite nonnegative Borel measure.
Let $s\in (1-n,1)$, $p:=\frac{n}{n-1+s}\in (1,\infty)$,
$w:=(M\mu)^{\frac{1}{p}}$, $\nu\in\{0,\frac{1}{3},\frac{2}{3}\}^n$,
$j\in \mathbb{N}$, $\mathcal G^\nu_j$ be as in \eqref{eq:Fj},
and $Q\in\mathcal G^\nu_j$.
Define
\begin{align*}
h_Q:=2^{-j\frac{n-1}{n}}\ell(Q)^{n-1},
\quad
w_Q:=\mathop{\rm ess\,inf}_{\partial^*E\cap Q^\circ}w,
\quad \text{and}
\quad
\lambda_Q:=2^{-j}\ell(Q)^{-s}.
\end{align*}
Then
\begin{align}\label{eq:local-perimeter}
h_Q \lesssim {\mathcal H}^{n-1}(\partial^*E\cap Q^\circ),\quad
\mu(Q)\lesssim \ell(Q)^nw_Q^p,
\end{align}
and
\begin{align}\label{eq:density-gain}
\lambda_Q\mu(Q)^{\frac{1}{p}}
\lesssim 2^{-\frac{j}{n}}h_Q w_Q,
\end{align}
where the implicit positive constants depend only on $n$.
\end{lemma}

\begin{proof}
By the assumption $Q\in\mathcal G^\nu_j$, we have
$m_Q\sim 2^{-j}$.
From this and
the relative isoperimetric inequality \eqref{eq:reiso}, it follows that
\begin{align*}
h_Q=2^{-j\frac{n-1}{n}}\ell(Q)^{n-1}
\lesssim
m_Q^{\frac{n-1}{n}}\ell(Q)^{n-1}
\lesssim
{\mathcal H}^{n-1}(\partial^*E\cap Q^\circ).
\end{align*}
Moreover, there exists a
positive constant $C_n$ depending only on $n$ such that,
for any $z\in Q$,
$Q\subset B(z,C_n\ell(Q))$.
Using this and the definitions of $M\mu$ and $w$,
we find that, for any $z\in Q$,
\begin{align*}
\mu(Q)\le
\mu \left(B(z,C_n\ell(Q))\right)
\le
C_n\ell(Q)^nM\mu(z)=C_n\ell(Q)^nw(z)^p.
\end{align*}
Taking the essential infimum over $z\in \partial^*E\cap Q^\circ$
proves \eqref{eq:local-perimeter}.
Finally, by \eqref{eq:local-perimeter}
and $p=\frac{n}{n-1+s}$, we have
\begin{align*}
\lambda_Q\mu(Q)^{\frac{1}{p}}
\lesssim
2^{-j}\ell(Q)^{-s}\ell(Q)^{\frac{n}{p}}w_Q
=2^{-j}\ell(Q)^{n-1} w_Q
=2^{-\frac{j}{n}}h_Q w_Q,
\end{align*}
which is \eqref{eq:density-gain}.
This completes the proof of Lemma~\ref{lem:local-charge}.
\end{proof}

Next, we prove Proposition~\ref{pro:11} by using the lemmas established above.

\begin{proof}[Proof of Proposition~\ref{pro:11}]
The case $\mu\equiv0$ is immediate. We now assume that
$\mu$ is not identically zero. Let $w:=(M\mu)^{\frac{1}{p}}$ and assume that
\begin{align*}
W:=\int_{\partial^*E}w\,d{\mathcal H}^{n-1}<\infty.
\end{align*}
Let $\mathcal G^\nu_j$ and $G^\nu_j$ be as in
\eqref{eq:Fj} and \eqref{eq:Gj-definition} for the grid ${\mathscr D}^\nu$.
By the Lebesgue density theorem
(see, for instance, \cite[Theorem 1.35]{EvansGariepy2015}),
we find that, for any cube $Q$,
$|E\cap Q|=|E^1\cap Q|$.
Applying Lemma~\ref{lem:dyadic-reduction} to $E^1$ and using this
identity, Lemma~\ref{lem:last-crossing}, and
Lemma~\ref{lem:zero-order-crossing}, we conclude that
\begin{align}\label{eq:ffmu}
\mathcal D^s{\mathbf 1}_{E^1}(x)
\lesssim\sum_{\nu\in\{0,\frac{1}{3},\frac{2}{3}\}^n}
\sum_{j\in{\mathbb N}}G_j^\nu(x)
\end{align}
holds for any $x\in E^0\cup E^1$ when $s\in (0,1)$ and
holds for any $x\in {\mathbb R}^n$ when $s=0$.
For any $\nu\in\{0,\frac{1}{3},\frac{2}{3}\}^n$,
$j\in {\mathbb N}$, and $k\in {\mathbb Z}$,
let
\begin{align*}
\mathcal{G}^\nu_{j,k}:= \left\{Q\in \mathcal{G}^{\nu}_{j}: \ell (Q)=2^{-k}\right\}.
\end{align*}
For each fixed $k$, the half-open cubes in $\mathcal{G}^\nu_{j,k}$
are disjoint. Then, for any $j\in{\mathbb N}$ and $x\in{\mathbb R}^n$,
\begin{align*}
G_j^\nu(x)&=
\sup_{k\in {\mathbb Z}}\sup_{Q\in\mathcal G_{j,k}^\nu}
2^{-j}\ell(Q)^{-s}{\mathbf 1}_Q(x)
= \sup_{k\in {\mathbb Z}}\sum_{Q\in\mathcal G_{j,k}^\nu}
2^{-j}\ell(Q)^{-s}{\mathbf 1}_Q(x)\\
&\le
\sum_{k\in {\mathbb Z}}\sum_{Q\in\mathcal G_{j,k}^\nu}
2^{-j}\ell(Q)^{-s}{\mathbf 1}_Q(x).
\end{align*}
Applying Lemma \ref{lem:HH} with $\mathcal F:=\mathcal{G}^\nu_{j,k}$,
Lemma \ref{lem:supercritical-counting} with $\mathcal F:=\mathcal{G}^\nu_{j}$,
and Lemma \ref{lem:local-charge}, we find that,
for any $\nu\in\{0,\frac{1}{3},\frac{2}{3}\}^n$ and $j\in {\mathbb N}$,
\begin{align}\label{eq:GWW}
\|G_j^\nu\|_{L^{p,1}(\mu)}
\le
\sum_{k\in {\mathbb Z}}
\left\|\sum_{Q\in\mathcal G_{j,k}^\nu}
2^{-j}\ell(Q)^{-s}{\mathbf 1}_Q\right\|_{L^{p,1}(\mu)}
\lesssim
2^{-\frac{j}{n}}
\int_{\partial^*E}w\,d{\mathcal H}^{n-1}
=2^{-\frac{j}{n}}W<\infty,
\end{align}
which implies that, for $\mu$-almost every $x\in {\mathbb R}^n$,
$G_j^\nu(x)<\infty$ and hence
\begin{align}\label{eq:mmmu}
\mu \left(\left\{x\in {\mathbb R}^n: G_j^\nu(x)=\infty\right\}\right)=0.
\end{align}

Assume for the moment that $s\in (0,1)$. We next claim that
\begin{align}\label{eq:essential-inclusion}
\partial^eE:=
{\mathbb R}^n\setminus(E^0\cup E^1)
\subset
\bigcup_{\nu\in\{0,\frac{1}{3},\frac{2}{3}\}^n}\bigcup_{j\in {\mathbb N}}
\{x:G_j^\nu(x)=\infty\}.
\end{align}
To do this, we
fix $x\in\partial^eE$.
For any $r\in (0,\infty)$, let
\begin{align*}
\rho(r):=\frac{|E\cap B(x,r)|}{|B(x,r)|}.
\end{align*}
Then there exist $\delta\in (0,\frac{1}{2})$ and
$\{r_m\}_{m\in \mathbb{N}}\subset (0,\infty)$
such that $r_m\downarrow0$ as $m\to\infty$ and, for any $m\in \mathbb{N}$,
\begin{align}\label{eq:RHO}
\delta\le\rho(r_m)\le1-\delta.
\end{align}
Indeed, otherwise $\min\{\rho(r),1-\rho(r)\}\to0$, which,
together with the continuity of $\rho$ on $(0,\infty)$ and
the definitions of $E^0$ and $E^1$, further implies that $x\in E^0\cup E^1$.
This contradicts the assumption $x\in\partial^eE =\mathbb{R}^n\setminus(E^0\cup E^1)$.

For any $m\in \mathbb{N}$, choose an adjacent dyadic cube $Q_m$
containing $B(x,r_m)$ and satisfying $\ell(Q_m)\le C_n r_m$.
From this, \eqref{eq:RHO}, and the definitions of $\rho$ and $m_Q$,
we infer that there exists a positive constant $c_n$,
depending only on $n$, such that, for any $m\in \mathbb{N}$,
$m_{Q_m}\ge c_n\delta.$
Let $J\in \mathbb{N}$ satisfy $2^{-J-1}< c_n\delta\le 2^{-J}$.
Then
\begin{align*}
\{Q_m\}_{m\in \mathbb{N}}\subset
\bigcup_{\nu\in\{0,\frac{1}{3},\frac{2}{3}\}^n}\bigcup_{j=1}^J
\mathcal G^{\nu}_{j}.
\end{align*}
Since there are only finitely many grids and only finitely many layers,
there exist $\nu_0$, $j_0$, and a subsequence $\{Q_{m_i}\}_{i \in \mathbb{N}}$ such that
$Q_{m_i}\in\mathcal G^{\nu_0}_{j_0}$ and $\ell(Q_{m_i})\to0$ as $i\to\infty$. Thus,
\begin{align*}
G_{j_0}^{\nu_0}(x)\ge2^{-j_0}\ell(Q_{m_i})^{-s}\to\infty
\end{align*}
as $i\to \infty$.
This shows \eqref{eq:essential-inclusion}.

Applying \eqref{eq:mmmu} and \eqref{eq:essential-inclusion},
we obtain $\mu(\partial^eE)=0$ when $s\in (0,1)$. Thus,
\eqref{eq:ffmu} holds for $\mu$-almost every $x\in{\mathbb R}^n$
in this case. For $s=0$, the same estimate already holds everywhere.
Combining this, \eqref{eq:ffmu}, and \eqref{eq:GWW}, we conclude that
\begin{align*}
\|\mathcal D^s{\mathbf 1}_{E^1}\|_{L^{p,1}(\mu)}
\lesssim \sum_{\nu\in\{0,\frac{1}{3},\frac{2}{3}\}^n}
\sum_{j\in{\mathbb N}}
\|G_j^\nu\|_{L^{p,1}(\mu)}
\lesssim
\left(\sum_{j\in{\mathbb N}}
2^{-\frac{j}{n}}\right)
W
\lesssim
W.
\end{align*}
In particular, when $s=0$, we have
$\|\mathcal D^0{\mathbf 1}_{E^1}\|_{L^{p,1}(\mu)}=p\mu(E^1)^{1/p}$,
which gives \eqref{eq:endpoint-iso}.
This then completes the proof
of \eqref{eq:characteristic} and hence Proposition~\ref{pro:11}.
\end{proof}

\section{Riesz Potential Estimates for Characteristic Functions}
\label{sec:alpha-greater-one}

The main purpose of this section is to establish the following
Lorentz estimate for Riesz
potentials of characteristic functions, which  will be used to prove
Theorem~\ref{thm:main} when $\alpha\in(1,n)$.

\begin{proposition}\label{pro:22}
Let $\beta\in(0,n-1)$ and $p:={\frac{n}{n-1-\beta}}$. Let $E\subset{\mathbb R}^n$
have finite perimeter and finite Lebesgue measure, and let $\mu$ be a
locally finite nonnegative Borel measure on ${\mathbb R}^n$. Then
\begin{align}\label{eq:potential-characteristic}
\|I_\beta{\mathbf 1}_E\|_{L^{p,1}(\mu)}
\le C_{n,\beta}\int_{\partial^*E}(M\mu)^{\frac{1}{p}}\,d{\mathcal H}^{n-1},
\end{align}
where $C_{n,\beta}$ depends only on $n$ and $\beta$.
\end{proposition}

To show Proposition~\ref{pro:22}, we need the following
pointwise estimate.
\begin{lemma}\label{lem:REpoint}
Let $\beta\in (0,n-1)$ and $E\subset \mathbb{R}^n$ be a Lebesgue measurable set with
$|E|<\infty$.
Then there exists a positive constant $C_{n,\beta}$,
depending only on $n$ and $\beta$, such that, for any $x\in{\mathbb R}^n$,
\begin{align}\label{eq:potential-low-density-domination}
I_\beta{\mathbf 1}_E(x)
\le C_{n,\beta}
\sum_{\nu\in\{0,\frac{1}{3},\frac{2}{3}\}^n}
\sum_{\genfrac{}{}{0pt}{}{Q\in{\mathscr D}^\nu}{0<e_Q\le{\frac{1}{2}}}}
\ell(Q)^\beta e_Q{\mathbf 1}_Q(x),
\end{align}
where $e_Q:={\frac{|E\cap Q|}{|Q|}}$ is
as in \eqref{eq:density}.
\end{lemma}

\begin{proof}
Fix $x\in{\mathbb R}^n$.
By the definition of the Riesz potential, we find that
\begin{align}\label{eq:IIB}
I_\beta{\mathbf 1}_E(x)
&=\frac{1}{\gamma(\beta)}
\int_{\mathbb{R}^n}\frac{{\mathbf 1}_{E}(y)}{|x-y|^{n-\beta}}\,dy\notag\\
&=\frac{1}{\gamma(\beta)}
\sum_{m\in\mathbb Z}
\int_{E\cap \{y\in{\mathbb R}^n:2^m\le |x-y|<2^{m+1}\}}|x-y|^{\beta-n}\,dy\notag\\
&\lesssim \sum_{m\in\mathbb Z}
2^{m(\beta-n)}|E\cap B(x,2^{m+1})|.
\end{align}
From
Lemma \ref{lemCub}(ii), we deduce that
there exist
$\nu_m\in\{0,\frac{1}{3},\frac{2}{3}\}^n$ and some cube
$Q_m\in{\mathscr D}^{\nu_m}$ such that
\begin{align*}
B(x,2^{m+1})\subset Q_m
\quad \text{and}\quad
\ell(Q_m)\sim 2^m
\end{align*}
with positive equivalence constants depending only on $n$.
Using this and $e_Q={\frac{|E\cap Q|}{|Q|}}$,
we conclude that, for any $m\in \mathbb{Z}$,
\begin{align*}
2^{m(\beta-n)}|E\cap B(x,2^{m+1})|
\le 2^{m(\beta-n)}|E\cap Q_m|
=2^{m(\beta-n)}e_{Q_m}\ell(Q_m)^n
\lesssim \ell(Q_m)^\beta e_{Q_m}.
\end{align*}
Since $\ell(Q_m)\sim2^m$, each selected cube can occur only
a bounded number of times, with a bound depending on $n$.
Summing the preceding estimate and using \eqref{eq:IIB}, we obtain
\begin{align}\label{eq:potential-all-dyadic-cubes}
I_\beta{\mathbf 1}_E(x)
\lesssim \sum_{\nu\in\{0,\frac{1}{3},\frac{2}{3}\}^n}
\sum_{Q\in{\mathscr D}^\nu}
\ell(Q)^\beta e_Q{\mathbf 1}_Q (x).
\end{align}
Fix $\nu\in\{0,\frac{1}{3},\frac{2}{3}\}^n$. Define
\begin{align*}
\mathcal J^\nu(x):=
\left\{Q\in{\mathscr D}^\nu:x\in Q,\ e_Q>{\frac{1}{2}}\right\}.
\end{align*}
If $\mathcal J^\nu(x)$ is empty, the contribution of the high-density
cubes is zero. Suppose therefore that $\mathcal J^\nu(x)\ne\emptyset$.
Since $|E|<\infty$, every $Q\in\mathcal J^\nu(x)$ satisfies
\begin{align*}
\frac{1}{2}<e_Q
=\frac{|E\cap Q|}{\ell(Q)^n}
\le\frac{|E|}{\ell(Q)^n},
\end{align*}
which implies
$\ell(Q)<(2|E|)^{{\frac{1}{n}}}<\infty.$
From this and the definition of ${\mathscr D}^\nu$, it follows that
there exists a unique maximal cube
$Q_0\in\mathcal J^\nu(x)$ such that, for any $Q\in \mathcal J^\nu(x)$, $Q\subset Q_0$.

Let $P$ be the dyadic parent of $Q_0$. Then
$x\in Q_0\subset P$ and $\ell(P)=2\ell(Q_0)$.
By the maximality of $Q_0$, we obtain
$P\notin \mathcal J^\nu(x)$ and hence
$e_P\le\frac{1}{2}.$
On the other hand, from $Q_0\subset P$ and $|P|=2^n|Q_0|$, we infer that
\begin{align*}
e_P
=\frac{|E\cap P|}{|P|}
\ge\frac{|E\cap Q_0|}{2^n|Q_0|}
=2^{-n}e_{Q_0}
>2^{-n-1}.
\end{align*}
Thus, $2^{-n-1}<e_P\le\frac{1}{2}.$
Applying this and the fact that
every cube in $\mathcal J^\nu(x)$ is a descendant of $Q_0$, we
conclude that
\begin{align*}
\sum_{Q\in\mathcal J^\nu(x)}\ell(Q)^\beta e_Q
\le\sum_{r=0}^\infty
\left(2^{-r}\ell(Q_0)\right)^\beta
=\frac{\ell(Q_0)^\beta}{1-2^{-\beta}}
=\frac{2^{-\beta}}{1-2^{-\beta}}\ell(P)^\beta
\le\frac{2^{n+1-\beta}}{1-2^{-\beta}}
\ell(P)^\beta e_P,
\end{align*}
which further implies that
\begin{align*}
\sum_{Q\in{\mathscr D}^\nu,\,e_Q>{\frac{1}{2}}}
\ell(Q)^\beta e_Q {\mathbf 1}_Q (x)
\lesssim
\sum_{Q\in{\mathscr D}^\nu,\, 0<e_Q\le{\frac{1}{2}}}
\ell(Q)^\beta e_Q {\mathbf 1}_Q (x).
\end{align*}
From this and
\eqref{eq:potential-all-dyadic-cubes},
we deduce \eqref{eq:potential-low-density-domination}.
This completes the proof of Lemma \ref{lem:REpoint}.
\end{proof}

Next, we prove Proposition~\ref{pro:22} by using Lemma
\ref{lem:REpoint}.

\begin{proof}[Proof of Proposition \ref{pro:22}]
The case $\mu\equiv0$ is immediate. We now assume that
$\mu$ is not identically zero.
Let $w:=(M\mu)^{\frac{1}{p}}$ and assume that
\begin{align*}
W:=\int_{\partial^*E}w\,d{\mathcal H}^{n-1}<\infty.
\end{align*}
Applying Lemma \ref{lem:REpoint}, we obtain
\begin{align}\label{eq:IEE}
\|I_\beta{\mathbf 1}_E\|_{L^{p,1}(\mu)}
&\lesssim\sum_{\nu\in\{0,\frac{1}{3},\frac{2}{3}\}^n}
\left\|
\sum_{\genfrac{}{}{0pt}{}{Q\in{\mathscr D}^\nu}{0<e_Q\le{\frac{1}{2}}}}
\ell(Q)^\beta e_Q{\mathbf 1}_Q
\right\|_{L^{p,1}(\mu)}.
\end{align}
Fix $\nu\in\{0,\frac{1}{3},\frac{2}{3}\}^n$.
For any $j\in \mathbb{N}$ and
$k\in\mathbb Z$, define
\begin{align*}
\mathcal G_j^+
:=\left\{Q\in{\mathscr D}^\nu:2^{-j-1}<e_Q\le2^{-j}\right\}
\quad\text{and}\quad
\mathcal G_{j,k}^+
:=\left\{Q\in\mathcal G_j^+:\ell(Q)=2^{-k}\right\}.
\end{align*}
If $Q\in\mathcal G_j^+$, then $e_Q\le2^{-j}\le{\frac{1}{2}}$ and
hence $m_Q=\min\{e_Q,1-e_Q\}=e_Q$. From this and
\eqref{eq:Fj}, it follows that
$\mathcal G_j^+\subset\mathcal G_j^\nu$, where $\mathcal G_j^\nu$ is as in \eqref{eq:Fj}.
For any $Q\in\mathcal G_j^+$, let
\begin{align*}
w_Q:=\mathop{\rm ess\,inf}_{\partial^*E\cap Q^\circ}w,
\quad
h_Q:=2^{-j\frac{n-1}{n}}\ell(Q)^{n-1},\quad
\text{and}
\quad \lambda_Q:=2^{-j}\ell(Q)^\beta.
\end{align*}
By this, the assumption $p={\frac{n}{n-1-\beta}}$,
and Lemma \ref{lem:local-charge} with $s:=-\beta$, we have
\begin{align}\label{eq:potential-density-gain}
\lambda_Q\mu(Q)^{\frac{1}{p}}
\lesssim
2^{-\frac{j}{n}}h_Q w_Q.
\end{align}
For fixed $j$ and $k$, the quantities $\lambda_Q$ have the common
value $2^{-j}2^{-k\beta}$ on $\mathcal G_{j,k}^+$. Moreover,
$\ell(Q)^\beta e_Q\le\lambda_Q$ on this family. Thus,
Lemma~\ref{lem:HH}, \eqref{eq:potential-density-gain}, and
Lemma~\ref{lem:supercritical-counting} give
\begin{align*}
\left\|
\sum_{\genfrac{}{}{0pt}{}{Q\in{\mathscr D}^\nu}{0<e_Q\le{\frac{1}{2}}}}
\ell(Q)^\beta e_Q{\mathbf 1}_Q
\right\|_{L^{p,1}(\mu)}
&\le
\sum_{j=1}^\infty\sum_{k\in{\mathbb Z}}
\left\|\sum_{Q\in\mathcal G_{j,k}^+}
\lambda_Q{\mathbf 1}_Q\right\|_{L^{p,1}(\mu)}\\
&\lesssim
\sum_{j=1}^\infty 2^{-{\frac{j}{n}}}\sum_{k\in{\mathbb Z}}
\left[\sum_{Q\in\mathcal G_{j,k}^+}(h_Qw_Q)^p\right]^{\frac{1}{p}}\\
&\lesssim W\sum_{j=1}^\infty2^{-{\frac{j}{n}}}
\lesssim W.
\end{align*}
This, combined with \eqref{eq:IEE}, implies
\eqref{eq:potential-characteristic}, and hence completes the proof
of Proposition~\ref{pro:22}.
\end{proof}

\section{Proofs of Theorem \ref{thm:main} and Its Corollaries}\label{sec:proofs}
In this section,
we present the proofs of the main results.
We first prove Theorem \ref{thm:main}
by using Propositions~\ref{pro:11} and \ref{pro:22}.

\begin{proof}[Proof of Theorem \ref{thm:main}]
Let $f\in C_{\rm c}^\infty({\mathbb R}^n)$, $u:=I_1f$, and
$w:=(M\mu)^{\frac{1}{p}}$. We may assume that
$\int_{{\mathbb R}^n}|Rf|w\,dx<\infty$.
Then $u\in C^\infty({\mathbb R}^n)$ and, for any $x\in{\mathbb R}^n$,
\begin{align}\label{eq:uuu}
|u(x)|\le C_f(1+|x|)^{1-n},
\end{align}
where $C_f$ is a positive constant depending on $f$.
For any $t\ne0$, define
\begin{align*}
A_t:=
\begin{cases}
\{u>t\}, & t\in(0,\infty),\\
\{u<t\}, & t\in(-\infty,0),
\end{cases}
\end{align*}
and let $A_0:=\emptyset$. Also, let $E_t:=\{|u|>t\}$ for any
$t\in(0,\infty)$. By \eqref{eq:uuu}, we find that the sets $A_t$ and $E_t$ are bounded.
From Sard's theorem (see, for instance, \cite[Theorem~6.10]{Lee2013}),
we infer that, for
almost every $t\in{\mathbb R}$,
$A_t$
has compact smooth boundary (possibly empty),
which further implies that $A_t$ has finite perimeter,
\begin{align}\label{eq:regular-level-properties}
A_t=A_t^1,
\quad\text{and}\quad
\partial^*A_t=\partial^*\{u>t\}=\{u=t\}.
\end{align}
Since $E_t=A_t\cup A_{-t}$, we deduce that $E_t$
has finite perimeter for almost every $t\in(0,\infty)$.
We then consider the following three cases for $\alpha$.

\emph{Case 1: $\alpha\in(0,1)$.}
In this case, let $s:=1-\alpha\in(0,1)$.
By the Fourier multiplier representations of Riesz potentials
and the fractional Laplacian
(see \cite{Stein1970} and \cite[Proposition~3.3]{DPV2012}),
we have, in the sense of tempered distributions,
\begin{align}\label{eq:IIFF}
I_\alpha f
=(-\Delta)^{\frac{s}{2}}I_1f
=(-\Delta)^{\frac{s}{2}}u.
\end{align}
Moreover, $Rf=\nabla I_1f=\nabla u$.
The singular integral representation of the fractional Laplacian is
\begin{align}\label{eq:lpls}
(-\Delta)^{\frac{s}{2}}u(x)
=c_{n,s}\,{\rm p.v.}\int_{{\mathbb R}^n}
\frac{u(x)-u(y)}{|x-y|^{n+s}}\,dy,
\end{align}
where $c_{n,s}$ is the standard positive normalization constant
depending only on $n$ and $s$.
Since $u\in C^\infty({\mathbb R}^n)$ is bounded and since $s\in(0,1)$,
it follows that, for any $x\in{\mathbb R}^n$,
\begin{align*}
\int_{{\mathbb R}^n}
\frac{|u(x)-u(y)|}{|x-y|^{n+s}}\,dy
&\le
\sup_{z\in\overline{B(x,1)}}|\nabla u(z)|
\int_{|x-y|<1}|x-y|^{1-n-s}\,dy\\
&\quad+
2\|u\|_{L^\infty({\mathbb R}^n)}
\int_{|x-y|\ge1}|x-y|^{-n-s}\,dy\\
&=n|B({\bf 0},1)|
\left[
\frac{1}{1-s}
\sup_{z\in\overline{B(x,1)}}|\nabla u(z)|
+\frac{2}{s}\|u\|_{L^\infty({\mathbb R}^n)}
\right]
<\infty.
\end{align*}
Thus, the integral in \eqref{eq:lpls} is absolutely convergent
and defines
a continuous function.
Since $I_\alpha f$ is also continuous,
the distributional identity \eqref{eq:IIFF} holds pointwise.

Moreover, by the definition of $\mathcal D^s$, we conclude that,
for any $x\in{\mathbb R}^n$,
\begin{align}\label{eq:fraclap-dominated}
|(-\Delta)^{\frac{s}{2}}u(x)|
\le c_{n,s}\mathcal D^su(x).
\end{align}
From this and \eqref{eq:IIFF}, we infer that
\begin{align}\label{eq:Ialpha}
\|I_\alpha f\|_{L^{p,1}(\mu)}
=
\|(-\Delta)^{\frac{s}{2}}u\|_{L^{p,1}(\mu)}
\le c_{n,s}\|\mathcal D^su\|_{L^{p,1}(\mu)}.
\end{align}
Observe that, for any $a,b\in{\mathbb R}$,
\begin{align*}
|a-b|
=
\int_0^\infty
\left|
{\mathbf 1}_{\{a>t\}}-{\mathbf 1}_{\{b>t\}}
\right|\,dt+
\int_{-\infty}^0
\left|
{\mathbf 1}_{\{a<t\}}-{\mathbf 1}_{\{b<t\}}
\right|\,dt.
\end{align*}
From this and the definition of $A_t$, we deduce that,
for any $x,y\in{\mathbb R}^n$,
\begin{align*}
|u(x)-u(y)|
=
\int_{\mathbb R}
\left|
{\mathbf 1}_{A_t}(x)-{\mathbf 1}_{A_t}(y)
\right|\,dt,
\end{align*}
which, together with the definition of $\mathcal D^s$ and
Tonelli's theorem, implies that, for any $x\in{\mathbb R}^n$,
\begin{align*}
\mathcal D^su(x)
=
\int_{\mathbb R}\mathcal D^s{\mathbf 1}_{A_t}(x)\,dt.
\end{align*}
By this, Minkowski's inequality,
\eqref{eq:regular-level-properties},
Proposition~\ref{pro:11}, and
Lemma~\ref{lem:coarea}, we obtain
\begin{align*}
\|\mathcal D^su\|_{L^{p,1}(\mu)}
&\le
\int_{\mathbb R}
\|\mathcal D^s{\mathbf 1}_{A_t}\|_{L^{p,1}(\mu)}\,dt=
\int_{\mathbb R}
\|\mathcal D^s{\mathbf 1}_{A_t^1}\|_{L^{p,1}(\mu)}\,dt\\
&\lesssim
\int_{\mathbb R}\int_{\partial^*A_t}
w\,d{\mathcal H}^{n-1}\,dt=
\int_{\mathbb R}\int_{\partial^*\{u>t\}}
w\,d{\mathcal H}^{n-1}\,dt\\
&=\int_{{\mathbb R}^n}|\nabla u(x)|w(x)\,dx.
\end{align*}
Combining this, \eqref{eq:Ialpha}, and
$|\nabla u|=|\nabla I_1f|=|Rf|$, we obtain
\begin{align*}
\|I_\alpha f\|_{L^{p,1}(\mu)}
\lesssim
\|\mathcal D^su\|_{L^{p,1}(\mu)}
\lesssim
\int_{{\mathbb R}^n}|Rf(x)|(M\mu(x))^{\frac{1}{p}}\,dx.
\end{align*}
This shows \eqref{eq:main} in this case.

\emph{Case 2: $\alpha=1$.}
In this case, since $u$ is continuous, $E_t$ is open and hence
$E_t\subset E_t^1$ for every $t>0$.
Applying this, \eqref{eq:Lorentz-distribution},
Proposition~\ref{pro:11} with $s=0$,
Lemma~\ref{lem:coarea},
and the fact that $|\nabla |u|(x)|= |\nabla u(x)|$
for almost every $x\in{\mathbb R}^n$, we obtain
\begin{align}\label{eq:ca2}
\|I_1f\|_{L^{p,1}(\mu)}
&=p\int_0^\infty\mu(E_t)^{\frac{1}{p}}\,dt
\le p\int_0^\infty\mu(E_t^1)^{\frac{1}{p}}\,dt
\notag\\
&\lesssim
\int_0^\infty\int_{\partial^*E_t}
w\,d{\mathcal H}^{n-1}\,dt
=
\int_{{\mathbb R}^n}|\nabla|u|(x)|w(x)\,dx\notag \\
&=\int_{{\mathbb R}^n}|\nabla u(x)|w(x)\,dx
=
\int_{{\mathbb R}^n}|Rf(x)|w(x)\,dx.
\end{align}
This proves \eqref{eq:main} when $\alpha=1$.

\emph{Case 3: $\alpha\in(1,n)$.}
In this case, let $\beta:=\alpha-1\in(0,n-1)$.
From the semigroup property of Riesz potentials
(see \cite[Section~1]{SSVS2017}) and from $R=\nabla I_1$,
it follows that
\begin{align}\label{eq:Riesz-semigroup-supercritical}
I_\alpha f=I_\beta I_1f=I_\beta u
\quad\text{and}\quad
\nabla u=\nabla I_1f=Rf.
\end{align}
By \eqref{eq:uuu}, we find that $I_\beta|u|(x)$ is finite
for any $x\in{\mathbb R}^n$ and hence the semigroup identity
\eqref{eq:Riesz-semigroup-supercritical} holds pointwise.
Since $I_\alpha f=I_\beta u$ and the kernel of $I_\beta$
is nonnegative, from the layer-cake representation
$|u(y)|=\int_0^\infty{\mathbf 1}_{E_t}(y)\,dt$ and Tonelli's
theorem, we infer that, for any $x\in{\mathbb R}^n$,
\begin{align*}
|I_\alpha f(x)|
&=|I_\beta u(x)|
\le I_\beta|u|(x)\\
&=\frac{1}{\gamma(\beta)}\int_{{\mathbb R}^n}
\frac{1}{|x-y|^{n-\beta}}
\left[\int_0^\infty{\mathbf 1}_{E_t}(y)\,dt\right]\,dy\\
&=\int_0^\infty I_\beta{\mathbf 1}_{E_t}(x)\,dt.
\end{align*}
Applying this, Minkowski's inequality, and
Proposition~\ref{pro:22}, we obtain
\begin{align}\label{eq:If2}
\|I_\alpha f\|_{L^{p,1}(\mu)}
\le
\int_0^\infty
\|I_\beta{\mathbf 1}_{E_t}\|_{L^{p,1}(\mu)}\,dt
\lesssim
\int_0^\infty\int_{\partial^*E_t}
w\,d{\mathcal H}^{n-1}\,dt.
\end{align}
Moreover, from Lemma~\ref{lem:coarea} and
\eqref{eq:Riesz-semigroup-supercritical}, we deduce that
\begin{align*}
\int_0^\infty\int_{\partial^*E_t}
w\,d{\mathcal H}^{n-1}\,dt
&=
\int_{{\mathbb R}^n}|\nabla|u|(x)|w(x)\,dx
=
\int_{{\mathbb R}^n}|\nabla u(x)|w(x)\,dx\\
&=
\int_{{\mathbb R}^n}|Rf(x)|w(x)\,dx.
\end{align*}
Combining this and \eqref{eq:If2}, we obtain
\eqref{eq:main} in this case.
This then completes the proof of Theorem \ref{thm:main}.
\end{proof}

Next, we show Corollary~\ref{cor:curl-free}.

\begin{proof}[Proof of Corollary~\ref{cor:curl-free}]
Let $F\in L^1({\mathbb R}^n;{\mathbb R}^n)$ satisfy
$\operatorname{curl}F=0$ in the sense of distributions.
We may assume that the right-hand side of
\eqref{eq:curl-free-measure} is finite.
The case $\mu=0$ is immediate.
We next assume that $\mu$ is nonzero
and compactly supported (that is, there exists a compact set $K$
such that $\mu({\mathbb R}^n \setminus K)=0$), and write $w:=(M\mu)^{\frac{1}{p}}$.
From the Coifman--Rochberg theorem
\cite[Proposition~2]{CoifmanRochberg1980} and $0<\frac{1}{p}<1$, it follows that
$w\in A_1$.

We now claim that, for any $u\in C_{\rm c}^\infty({\mathbb R}^n)$,
\begin{align}\label{eq:curl-free-smooth}
\|I_\alpha\nabla u\|_{L^{p,1}(\mu;{\mathbb R}^n)}
\lesssim \int_{{\mathbb R}^n}|\nabla u(x)|w(x)\,dx,
\end{align}
where the implicit positive constant depends only on $n$ and $\alpha$.
When $\alpha \in (0,1)$, using integration by parts, we obtan
\begin{align*}
I_\alpha\nabla u(x)
=\frac{n-\alpha}{\gamma(\alpha)}
\int_{{\mathbb R}^n}
\frac{(u(x)-u(y))(x-y)}{|x-y|^{n+2-\alpha}}\,dy.
\end{align*}
Since $u\in C_{\rm c}^\infty({\mathbb R}^n)$,
we infer that this integral is absolutely convergent.
Consequently, for any $x\in {\mathbb R}^n$,
\begin{align}\label{eq:DDU}
|I_\alpha\nabla u(x)|\lesssim \mathcal D^{1-\alpha}u(x).
\end{align}
Repeating the proof of Theorem \ref{thm:main} with \eqref{eq:fraclap-dominated}
replaced by \eqref{eq:DDU}, we obtain \eqref{eq:curl-free-smooth}.

When $\alpha\in (1,n)$, integration by parts gives, for any $x\in{\mathbb R}^n$,
\begin{align*}
|I_\alpha\nabla u(x)|
\lesssim\int_{{\mathbb R}^n}\frac{|u(y)|}{|x-y|^{n+1-\alpha}}\,dy
\lesssim I_{\alpha-1}|u|(x).
\end{align*}
Applying Proposition~\ref{pro:22} to the superlevel sets of $|u|$
and then the weighted coarea formula, as in Case 3 of the proof of
Theorem~\ref{thm:main}, proves \eqref{eq:curl-free-smooth}.

It remains to consider $\alpha=1$.
For any set $E$ of finite perimeter and finite Lebesgue measure
and for any $x\in{\mathbb R}^n$, define
\begin{align*}
T_E(x):=\sum_{\nu\in\{0,\frac{1}{3},\frac{2}{3}\}^n}
\sum_{Q\in{\mathscr D}^\nu}\min\left\{\frac{|E\cap Q|}{|Q|},
\frac{|Q\setminus E|}{|Q|}\right\}{\mathbf 1}_Q(x).
\end{align*}
Using Lemmas~\ref{lem:HH}, \ref{lem:supercritical-counting},
and \ref{lem:local-charge} with $s=0$, and summing over $j$ and $k$, we obtain
\begin{align}\label{eq:zero-order-density-sum}
\|T_E\|_{L^{p,1}(\mu)}
&\le\sum_{\nu\in\{0,\frac{1}{3},\frac{2}{3}\}^n}\sum_{j\in{\mathbb N}}
\sum_{k\in{\mathbb Z}}
\left\|\sum_{Q\in\mathcal G_{j,k}^\nu}2^{-j}{\mathbf 1}_Q
\right\|_{L^{p,1}(\mu)}\notag\\
&\lesssim\sum_{j\in{\mathbb N}}2^{-\frac{j}{n}}
\int_{\partial^*E}w\,d{\mathcal H}^{n-1}
\lesssim\int_{\partial^*E}w\,d{\mathcal H}^{n-1}.
\end{align}
For any $z\ne0$, let $K(z):=\nabla_z[\gamma(1)^{-1}|z|^{1-n}]$.
For any $x\in{\mathbb R}^n$ and $k\in{\mathbb Z}$, let
$$A_k(x):=\left\{y:2^k\le|x-y|<2^{k+1}\right\}.$$
Then $|K(z)|\lesssim|z|^{-n}$ and,
for any $k\in{\mathbb Z}$ and $x\in{\mathbb R}^n$,
\begin{align}\label{eq:KKK}
\int_{A_k(x)}K(x-y)\,dy=0.
\end{align}
By Lemma~\ref{lemCub}(ii), we find that there exist
$\nu_k\in\{0,\frac{1}{3},\frac{2}{3}\}^n$
and
a cube
$Q_k\in{\mathscr D}^{\nu_k}$ such that $B(x,2^{k+1})\subset Q_k$ and
$\ell(Q_k)\sim2^k$.
From this and \eqref{eq:KKK}, we deduce that, for any $x\in {\mathbb R}^n$,
\begin{align}\label{eq:QQK}
\left|\int_{A_k(x)}K(x-y){\mathbf 1}_E(y)\,dy\right|
&\lesssim 2^{-kn}
\min\{|E\cap A_k(x)|,|A_k(x)\setminus E|\}\notag\\
&\lesssim \min\left\{\frac{|E\cap Q_k|}{|Q_k|},
\frac{|Q_k\setminus E|}{|Q_k|}\right\}.
\end{align}
For any $t\in(0,\infty)$, let
\begin{align*}
E_t^+:=\{y\in{\mathbb R}^n:u(y)>t\}
\quad \text{and}\quad
E_t^-:=\{y\in{\mathbb R}^n:u(y)<-t\}.
\end{align*}
Using the signed layer-cake representation
\begin{align*}
u(y)=\int_0^\infty
\left[{\mathbf 1}_{E_t^+}(y)-{\mathbf 1}_{E_t^-}(y)\right]\,dt
\end{align*}
and the fact that, for any $x\in{\mathbb R}^n$ and $k\in{\mathbb Z}$,
\begin{align*}
&\int_{A_k(x)}|K(x-y)|
\int_0^\infty
\left[{\mathbf 1}_{E_t^+}(y)+{\mathbf 1}_{E_t^-}(y)\right]\,dt\,dy\\
&\qquad=
\int_{A_k(x)}|K(x-y)||u(y)|\,dy<\infty,
\end{align*}
we can apply Fubini's theorem on each annulus.
By the triangle inequality and Tonelli's theorem,
we conclude that,
for any $x\in{\mathbb R}^n$,
\begin{align*}
|I_1\nabla u(x)|&=|Ru(x)|
\le\sum_{k\in{\mathbb Z}}
\left|\int_{A_k(x)}K(x-y)u(y)\,dy\right|\\
&\le\sum_{k\in{\mathbb Z}}\int_0^\infty
\left[
\left|\int_{A_k(x)}
K(x-y){\mathbf 1}_{E_t^+}(y)\,dy\right|
+
\left|\int_{A_k(x)}
K(x-y){\mathbf 1}_{E_t^-}(y)\,dy\right|
\right]\,dt\\
&=\int_0^\infty\sum_{k\in{\mathbb Z}}
\left[
\left|\int_{A_k(x)}
K(x-y){\mathbf 1}_{E_t^+}(y)\,dy\right|
+
\left|\int_{A_k(x)}
K(x-y){\mathbf 1}_{E_t^-}(y)\,dy\right|
\right]\,dt\\
&\lesssim\int_0^\infty
\left[T_{E_t^+}(x)+T_{E_t^-}(x)\right]\,dt.
\end{align*}
Here, the last inequality follows from \eqref{eq:QQK}
and the fact that $\ell(Q_k)\sim2^k$, which ensures
that each selected cube occurs only a bounded number
of times with a bound depending only on $n$.
Applying this,
Minkowski's inequality, \eqref{eq:zero-order-density-sum}, Lemma \ref{lem:coarea} to
$u^+:=\max\{u,0\}$ and $u^-:=\max\{-u,0\}$, and the fact that
$|\nabla u^+|+|\nabla u^-|=|\nabla u|$ almost everywhere,
we conclude that
\begin{align*}
\|I_1\nabla u\|_{L^{p,1}(\mu;{\mathbb R}^n)}
&\lesssim\int_0^\infty
\left[\int_{\partial^*E_t^+}w\,d{\mathcal H}^{n-1}
+\int_{\partial^*E_t^-}w\,d{\mathcal H}^{n-1}\right]\,dt\\
&=\int_{{\mathbb R}^n}
\left[|\nabla u^+(x)|+|\nabla u^-(x)|\right]w(x)\,dx
=\int_{{\mathbb R}^n}|\nabla u(x)|w(x)\,dx.
\end{align*}
This completes the proof of \eqref{eq:curl-free-smooth} for all
$\alpha\in(0,n)$.

We now pass from smooth gradients to the given field $F$.
By an approximation result of Bonami and
Poornima \cite[Lemma~1]{BonamiPoornima1987}, as used in the proof of
\cite[Theorem~1.1]{Spector2020}, we conclude that there exist functions
$\{\varphi_k\}_{k\in {\mathbb N}}\subset C_{\rm c}^\infty({\mathbb R}^n)$ such that
$
\nabla\varphi_k\to  F$
in $L^1({\mathbb R}^n;{\mathbb R}^n)$.
Applying the classical Sobolev inequality to
$\varphi_k-\varphi_\ell$ (see, for instance, \cite[Theorem 4.8]{EvansGariepy2015}),
we obtain, as $k,\ell\to\infty$,
\begin{align*}
\|\varphi_k-\varphi_\ell\|_{L^{\frac{n}{n-1}}({\mathbb R}^n)}
\le C_n
\|\nabla\varphi_k-\nabla\varphi_\ell\|_{L^1({\mathbb R}^n;{\mathbb R}^n)}
\to 0,
\end{align*}
where $C_n$ is a positive constant depending only on $n$.
Thus, $\{\varphi_k\}_{k\in {\mathbb N}}$ is a
Cauchy sequence in $L^{\frac{n}{n-1}}({\mathbb R}^n)$.
From completeness, we infer that there exists
 $\Phi \in L^{\frac{n}{n-1}}({\mathbb R}^n)$
such that $\varphi_k\to \Phi$ in $L^{\frac{n}{n-1}}({\mathbb R}^n)$, and hence in
$L_{\mathrm{loc}}^1({\mathbb R}^n)$.
For any $i\in\{1,\ldots,n\}$ and $\psi\in C_{\rm c}^\infty({\mathbb R}^n)$,
using integration by parts, we have
\begin{align*}
\int_{{\mathbb R}^n}\Phi(x)\,\partial_i\psi(x)\,dx
&=\lim_{k\to\infty}
\int_{{\mathbb R}^n}\varphi_k(x)\,\partial_i\psi(x)\,dx\\
&=-\lim_{k\to\infty}
\int_{{\mathbb R}^n}\partial_i\varphi_k(x)\,\psi(x)\,dx
=-\int_{{\mathbb R}^n}F_i(x)\,\psi(x)\,dx.
\end{align*}
Consequently, $\nabla \Phi=F$
in the sense of distributions, and
$\Phi\in W_{\mathrm{loc}}^{1,1}({\mathbb R}^n)$.
Moreover, by the classical Sobolev inequality again, we have
\begin{align*}
\|\Phi\|_{L^{\frac{n}{n-1}}({\mathbb R}^n)}
= \lim_{k\to \infty} \|\varphi_k\|_{L^{\frac{n}{n-1}}({\mathbb R}^n)}
\lesssim \lim_{k\to \infty} \|\nabla\varphi_k\|_{L^1({\mathbb R}^n;{\mathbb R}^n)}
= \|F\|_{L^1({\mathbb R}^n;{\mathbb R}^n)}.
\end{align*}

We next show that there exists a sequence
$\{\Phi_k\}_{k\in {\mathbb N}}\subset  C_{\rm c}^\infty({\mathbb R}^n)$ such that
\begin{align}\label{eq:curl-free-weighted-approximation}
\int_{{\mathbb R}^n}|\nabla \Phi_k(x)-F(x)|w(x)\,dx\to 0
\quad\text{as }k\to\infty.
\end{align}
Choose $R_0>0$ such that $\operatorname{supp}\mu\subset B({\bf 0},R_0)$.
Fix $x\in{\mathbb R}^n$ with $|x|\ge2R_0$.
If $0<r\le |x|-R_0$, then $B(x,r)$ does not intersect
$B({\bf 0},R_0)$, and hence $\mu(B(x,r))=0$.
If $r>|x|-R_0$, then $r>{\frac{|x|}{2}}$. Thus,
\begin{align*}
\frac{\mu(B(x,r))}{|B(x,r)|}
\le \frac{\mu({\mathbb R}^n)}{\omega_n r^n}
\le \frac{2^n}{\omega_n}
\frac{\mu({\mathbb R}^n)}{|x|^n},
\end{align*}
where $\omega_n:=|B({\bf 0},1)|$.
From this, the definition of $M\mu$, and ${\frac{n}{p}}=n-\alpha$,
we deduce that, for any $x\in{\mathbb R}^n$ with $|x|\ge2R_0$,
\begin{align}\label{eq:wMu}
w(x)=[M\mu(x)]^{{\frac{1}{p}}}
\lesssim \mu({\mathbb R}^n)^{{\frac{1}{p}}}|x|^{-{\frac{n}{p}}}
=\mu({\mathbb R}^n)^{{\frac{1}{p}}}|x|^{\alpha-n}.
\end{align}
Let $\eta\in C_{\rm c}^\infty(B({\bf 0},2))$ satisfy $0\le\eta\le1$ and
$\eta=1$ on $B({\bf 0},1)$.
For any $R\in (0,\infty)$, define
$\eta_R(\cdot):=\eta({\frac{\cdot}{R}})$ and $A_R:=B({\bf 0},2R)\setminus B({\bf 0},R)$.
Then
$\eta_R=1$ on $B({\bf 0},R)$,
$\operatorname{supp}\nabla\eta_R\subset A_R$, and
$|\nabla\eta_R|\le {\frac{C}{R}}$.
By $\Phi\in W_{\mathrm{loc}}^{1,1}({\mathbb R}^n)$ and $\nabla \Phi=F$,
we find that, for almost every $x\in {\mathbb R}^n$,
\begin{align*}
\nabla(\eta_R\Phi)(x)-F(x)
&=\eta_R(x)\nabla \Phi(x)+\Phi(x)\nabla\eta_R(x)-F(x)\\
&=(\eta_R-1)(x)F(x)+\Phi(x)\nabla\eta_R (x).
\end{align*}
From this, \eqref{eq:wMu}, and H\"older's inequality,
it follows that, for any $R\ge2R_0$,
\begin{align*}
\int_{{\mathbb R}^n}|\nabla(\eta_R\Phi)-F|w\,dx
&\le\int_{{\mathbb R}^n\setminus B({\bf 0},R)}|F|w\,dx
+\frac{C}{R}\int_{A_R}|\Phi|w\,dx\\
&\le\int_{{\mathbb R}^n\setminus B({\bf 0},R)}|F|w\,dx
+C\mu({\mathbb R}^n)^{\frac{1}{p}}R^{\alpha-n}
\|\Phi\|_{L^{\frac{n}{n-1}}(A_R)}
\to 0
\end{align*}
as $R\to \infty$.
Using this convergence and the fact that $w\in A_1$,
we choose sequences $\{R_k\}_{k\in\mathbb N}$ and
$\{\epsilon_k\}_{k\in\mathbb N}$ of positive numbers, with
$R_k\to\infty$ and $\epsilon_k\to0$, and define, for any $k\in\mathbb N$,
$\Phi_k:=(\eta_{R_k}\Phi)*\rho_{\epsilon_k},$
where $\rho_\epsilon(\cdot):=\epsilon^{-n}\rho(\cdot/\epsilon)$
for a nonnegative function $\rho\in C_{\rm c}^\infty(B(0,1))$
satisfying $\int_{{\mathbb R}^n}\rho(x)\,dx=1$.
For each $k$, we first choose $R_k>k$ so that the
cutoff error is less than $2^{-k}$, and then choose
$\epsilon_k\in(0,1/k)$ so that the mollification
error is less than $2^{-k}$ in the corresponding
norm $\|\cdot\|_{L^1(w\,dx; {\mathbb R}^n)}$. The resulting sequence
$\{\Phi_k\}_{k\in\mathbb N}\subset C_{\rm c}^\infty({\mathbb R}^n)$
satisfies \eqref{eq:curl-free-weighted-approximation}.

For any $r\in(0,\infty)$ and almost every $y\in{\mathbb R}^n$,
Tonelli's theorem yields
\begin{align*}
\int_{B(y,r)}\frac{d\mu(x)}{|x-y|^{n-\alpha}}
&=r^{\alpha-n}\mu(B(y,r))
+(n-\alpha)\int_0^r
\frac{\mu(B(y,t))}{t^{n-\alpha+1}}\,dt\\
&\lesssim M\mu(y)
\left[r^\alpha+\int_0^r t^{\alpha-1}\,dt\right]
\lesssim r^\alpha M\mu(y),
\end{align*}
which implies that
\begin{align*}
\int_{{\mathbb R}^n}\frac{d\mu(x)}{|x-y|^{n-\alpha}}
&=\int_{B(y,r)}\frac{d\mu(x)}{|x-y|^{n-\alpha}}
+\int_{{\mathbb R}^n\setminus B(y,r)}
\frac{d\mu(x)}{|x-y|^{n-\alpha}}\\
&\lesssim r^\alpha M\mu(y)
+r^{\alpha-n}\mu({\mathbb R}^n).
\end{align*}
Choosing
$r:=[\mu({\mathbb R}^n)/M\mu(y)]^{\frac{1}{n}}$,
we find that, for almost every $y\in{\mathbb R}^n$,
\begin{align*}
\int_{{\mathbb R}^n}\frac{d\mu(x)}{|x-y|^{n-\alpha}}
\lesssim
\mu({\mathbb R}^n)^{\frac{\alpha}{n}}
[M\mu(y)]^{1-\frac{\alpha}{n}}
=\mu({\mathbb R}^n)^{\frac{\alpha}{n}}w(y).
\end{align*}
By this and Tonelli's theorem, we conclude that,
for any $G\in L^1(w\,dx;{\mathbb R}^n)$,
\begin{align}\label{eq:curl-free-L1-mu}
\int_{{\mathbb R}^n}I_\alpha(|G|)(x)\,d\mu(x)
\lesssim \mu({\mathbb R}^n)^{{\frac{\alpha}{n}}}
\int_{{\mathbb R}^n}|G(x)|w(x)\,dx.
\end{align}
In particular, the integral defining $I_\alpha F(x)$ is absolutely convergent
$\mu$-almost everywhere. Applying
\eqref{eq:curl-free-L1-mu} with $G:=\nabla \Phi_k-F$ and
applying
\eqref{eq:curl-free-weighted-approximation}, we obtain
\begin{align*}
\|I_\alpha\nabla \Phi_k-I_\alpha F\|_{L^1(\mu;{\mathbb R}^n)}
\lesssim \mu({\mathbb R}^n)^{{\frac{\alpha}{n}}}
\|\nabla \Phi_k-F\|_{L^1(w\,dx;{\mathbb R}^n)}
\to 0.
\end{align*}
By the subsequence principle for $L^1$ convergence,
there exists a subsequence $\{\Phi_{k_i}\}_{i\in {\mathbb N}}$
such that, for $\mu$-almost every $x\in{\mathbb R}^n$,
$I_\alpha\nabla \Phi_{k_i}(x)\to I_\alpha F(x)$
as $i\to \infty$.
Using this convergence, \eqref{eq:Lorentz-distribution},
Fatou's lemma, \eqref{eq:curl-free-weighted-approximation},
and \eqref{eq:curl-free-smooth}, we find that
\begin{align*}
\|I_\alpha F\|_{L^{p,1}(\mu;{\mathbb R}^n)}
&=p\int_0^\infty
\mu\left(\{|I_\alpha F|>t\}\right)^{{\frac{1}{p}}}\,dt\\
&\le p\int_0^\infty
\liminf_{i\to\infty}
\mu\left(\{|I_\alpha\nabla \Phi_{k_i}|>t\}\right)^{{\frac{1}{p}}}\,dt\\
&\le\liminf_{i\to\infty}
p\int_0^\infty
\mu\left(\{|I_\alpha\nabla \Phi_{k_i}|>t\}\right)^{{\frac{1}{p}}}\,dt\\
&=\liminf_{i\to\infty}
\|I_\alpha\nabla \Phi_{k_i}\|_{L^{p,1}(\mu;{\mathbb R}^n)}\\
&\lesssim
\lim_{i\to\infty}\int_{{\mathbb R}^n}|\nabla \Phi_{k_i}|w\,dx
=\int_{{\mathbb R}^n}|F|w\,dx.
\end{align*}
This shows \eqref{eq:curl-free-measure} for a compactly supported measure $\mu$.

For a general locally finite measure $\mu$,
let $\mu_m(\cdot):=\mu(B({\bf 0},m)\cap \cdot)$
for any $m\in {\mathbb N}$.
From this, $M\mu_m\le M\mu$, and the monotone convergence theorem,
we deduce \eqref{eq:curl-free-measure}.
This completes the proof of Corollary \ref{cor:curl-free}.
\end{proof}

Next, we give the proof of Corollary~\ref{cor:fractional-gradient}.

\begin{proof}[Proof of Corollary \ref{cor:fractional-gradient}]
Without loss of generality, we may assume that
the right-hand side of \eqref{eq:fractional-gradient-Lorentz} is finite.
When $\alpha=1$, the level-set and coarea argument in
\eqref{eq:ca2}, applied directly to $u$, gives
\eqref{eq:fractional-gradient-Lorentz}.
When $\alpha\in (0,1)$, let $s:=1-\alpha\in(0,1)$,
$w:=(M\mu)^{\frac{1}{p}}$, $u\in C_{\rm c}^\infty({\mathbb R}^n)$,
and $v:=I_su=I_{1-\alpha}u$.
Since $u\in C_{\rm c}^\infty({\mathbb R}^n)$,
from the definition of the Riesz potential, we infer that
 $v\in C^\infty({\mathbb R}^n)\cap L^\infty({\mathbb R}^n)$
and, for any $x\in{\mathbb R}^n$,
\begin{align*}
|v(x)|\le C_u(1+|x|)^{s-n}
\quad \text{and}\quad
|\nabla v(x)|\le C_u(1+|x|)^{s-n-1},
\end{align*}
where $C_u$ is a positive constant depending only on $n$, $\alpha$,
and $u$. In particular, $v$ tends to zero at infinity and
$\nabla v=\nabla^\alpha u$.
Moreover, the identity $(-\Delta)^{\frac{s}{2}}I_su=u$ holds in
the sense of tempered distributions. As in the proof of
\eqref{eq:fraclap-dominated}, the boundedness and smoothness of $v$
and $s\in(0,1)$ ensure that the singular integral below is absolutely
convergent and continuous on ${\mathbb R}^n$. Thus, the identity also holds
pointwise, and
\begin{align}\label{eq:nnns}
|u(x)|
&=\left|c_{n,s}\int_{{\mathbb R}^n}
\frac{v(x)-v(y)}{|x-y|^{n+s}}\,dy\right|
\le c_{n,s}\mathcal D^sv(x).
\end{align}
Repeating the argument in Case 1 of the proof of
Theorem~\ref{thm:main}, with $u$ replaced by $v$ and
\eqref{eq:fraclap-dominated} replaced by \eqref{eq:nnns},
we obtain \eqref{eq:fractional-gradient-Lorentz}.
This completes the proof of
Corollary~\ref{cor:fractional-gradient}.
\end{proof}

Applying Corollary~\ref{cor:fractional-gradient},
we now prove Corollary~\ref{cor:fractional-hardy-pole}.

\begin{proof}[Proof of Corollary \ref{cor:fractional-hardy-pole}]
Let $a\in{\mathbb R}^n$ satisfy $M\mu(a)<\infty$, and let
$\omega_n:=|B({\bf 0},1)|$, $p:={\frac{n}{n-\alpha}}$, and
$p':={\frac{n}{\alpha}}$.
By the definition of $M\mu$, for any $r\in(0,\infty)$, we have
\begin{align*}
\mu(\{a\})\le\mu(B(a,r))
\le\omega_nr^nM\mu(a).
\end{align*}
Letting $r\downarrow0$ gives $\mu(\{a\})=0$.
Define $h_a(x):=|x-a|^{-\alpha}$ for any $x\ne a$, and $h_a(a):=0$.
For any $t\in(0,\infty)$, it follows that
\begin{align*}
\mu(\{x:h_a(x)>t\})
=\mu\left(B\left(a,t^{-{\frac{1}{\alpha}}}\right)\right).
\end{align*}
From this and \eqref{eq:weak-Lorentz-distribution}, we infer that
\begin{align*}
\|h_a\|_{L^{p',\infty}(\mu)}
=\sup_{t\in(0,\infty)}
t\mu\left(B\left(a,t^{-{\frac{1}{\alpha}}}\right)\right)^{\frac{\alpha}{n}}
=\sup_{r\in(0,\infty)}
r^{-\alpha}\mu(B(a,r))^{\frac{\alpha}{n}}
=\left[\omega_nM\mu(a)\right]^{\frac{\alpha}{n}}.
\end{align*}
Using this, ${\frac{1}{p}}+{\frac{1}{p'}}=1$, the Lorentz--H\"older
inequality, and Corollary~\ref{cor:fractional-gradient}, we find that
\begin{align*}
\int_{{\mathbb R}^n}\frac{|u(x)|}{|x-a|^\alpha}\,d\mu(x)
&=\int_{{\mathbb R}^n}|u(x)|h_a(x)\,d\mu(x)
\le\|u\|_{L^{p,1}(\mu)}\|h_a\|_{L^{p',\infty}(\mu)}\\
&\le C_{n,\alpha}[M\mu(a)]^{\frac{\alpha}{n}}
\int_{{\mathbb R}^n}|\nabla^\alpha u(x)|
[M\mu(x)]^{1-{\frac{\alpha}{n}}}\,dx.
\end{align*}
This gives \eqref{eq:fractional-hardy-pole} and
completes the proof of Corollary~\ref{cor:fractional-hardy-pole}.
\end{proof}

To prove Corollary~\ref{cor:A1},
we need the following proposition.

\begin{proposition}
\label{prop:A1-sharpness}
Let $\alpha \in (0,n)$ and $p={\frac{n}{n-\alpha}}$.
Suppose that there exist some $\theta\in{\mathbb R}$ and a positive constant $C$, independent
of $v$ and $f$, such that, for any $v\in A_1$ and $f\in C_{\rm c}^\infty({\mathbb R}^n)$,
\begin{align}\label{eq:Shap}
\|I_\alpha f\|_{L^{p,1}(v)}
\le C[v]_{A_1}^{\theta}
\int_{{\mathbb R}^n}|Rf(x)|v(x)^{\frac{1}{p}}\,dx.
\end{align}
Then $\theta\ge \frac{1}{p}$.
\end{proposition}

\begin{proof}
For any $\delta\in (0,\frac{n}{2})$, define
$v_\delta(\cdot):=|\cdot|^{-n+\delta}$.
From \cite[Example 7.1.7]{GrafakosCFA}, we deduce that
$v_\delta\in A_1$, $[v_\delta]_{A_1}\sim \delta^{-1}$ with
positive
equivalence constants depending on $n$,
and
\begin{align*}
v_\delta(B({\bf 0},1)):= \int_{B({\bf 0},1)}v_\delta(x)\,dx
=\frac{n|B({\bf 0},1)|}{\delta}.
\end{align*}
Choose a fixed nonnegative function $f\in C_{\rm c}^\infty({\mathbb R}^n)$,
independent of $\delta$, such that
\begin{align*}
\operatorname{supp}f\subset\{y\in{\mathbb R}^n:2<|y|<3\}
\quad
\text{and}
\quad
\int_{{\mathbb R}^n}f(y)\,dy=1.
\end{align*}
For any $x\in B({\bf 0},1)$ and $y\in\operatorname{supp}f$, we have
$|x-y|<4$, and hence
\begin{align*}
I_\alpha f(x)
=
\frac{1}{\gamma(\alpha)}
\int_{{\mathbb R}^n}|x-y|^{\alpha-n}f(y)\,dy
\ge \frac{4^{\alpha-n}}{\gamma(\alpha)}
=:c_{n,\alpha}>0.
\end{align*}
By this and the definition of the Lorentz norm,
we conclude that
\begin{align}\label{eq:A1-sharpness-lower}
\|I_\alpha f\|_{L^{p,1}(v_\delta)}
&=
p\int_{0}^\infty v_{\delta}(\{I_\alpha f>t\})^{\frac{1}{p}}\,dt
\ge
p\int_{0}^{c_{n,\alpha}} v_{\delta}(B({\bf 0},1))^{\frac{1}{p}}\,dt\notag\\
&=
p c_{n,\alpha}
\left[\frac{n|B({\bf 0},1)|}{\delta}\right]^{\frac{1}{p}}
\gtrsim \delta^{-\frac{1}{p}}.
\end{align}
Moreover, since $f\in C_{\rm c}^\infty ({\mathbb R}^n)$,
it follows that there exists a positive constant $C_f$,
depending only on $f$, such that,
for any $x\in{\mathbb R}^n$,
\begin{align*}
|Rf(x)|\le C_f(1+|x|)^{-n}.
\end{align*}
Using this, polar coordinates, ${\frac{n}{p}}=n-\alpha$,
and $\delta\in (0,\frac{n}{2})$,
we obtain
\begin{align*}
\int_{{\mathbb R}^n}|Rf(x)|v_\delta(x)^{\frac{1}{p}}\,dx
&\lesssim
\int_0^1 r^{\alpha-1+\frac{\delta}{p}}\,dr
+\int_1^\infty r^{-1-\frac{n-\delta}{p}}\,dr
\\
&=
\frac{1}{\alpha+{\frac{\delta}{p}}}
+\frac{p}{n-\delta}
\le
\frac{1}{\alpha}+\frac{2p}{n}.
\end{align*}
Combining this, \eqref{eq:Shap}, and \eqref{eq:A1-sharpness-lower},
we conclude that, for any $\delta\in (0,\frac{n}{2})$,
$\delta^{-\frac{1}{p}}\le C'\delta^{-\theta},$
where $C'$ is independent of $\delta$.
Letting $\delta\downarrow0$, we have $\theta\ge \frac{1}{p}$.
This completes the proof of Proposition \ref{prop:A1-sharpness}.
\end{proof}

Finally, we give the proof of Corollary~\ref{cor:A1}.
\begin{proof}[Proof of Corollary \ref{cor:A1}]
By the definition of $[v]_{A_1}$,
we have $Mv \lesssim [v]_{A_1}v$ almost everywhere.
Applying Theorem~\ref{thm:main} with $\mu:=v\,dx$ gives
\eqref{eq:weight}.
Furthermore, applying Proposition \ref{prop:A1-sharpness},
we find that the exponent $\frac{1}{p}$ on $[v]_{A_1}$ is optimal.
This completes the proof of Corollary \ref{cor:A1}.
\end{proof}

\smallskip

\noindent\textbf{Acknowledgements}\quad
The authors acknowledge the use of AI tools during the
exploratory stage of this project.
All mathematical arguments and proofs in the final manuscript were checked
and written by the authors.

%
%
%

\bigskip

\noindent Pingxu Hu, Dachun Yang and
Wen Yuan

\medskip

\noindent Laboratory of Mathematics and Complex Systems
(Ministry of Education of China),
School of Mathematical Sciences, Institute for Advanced Study,
Beijing Normal University,
Beijing 100875, The People's Republic of China

\smallskip

\noindent {\it E-mails}: \texttt{pingxuhu@mail.bnu.edu.cn} (P. Hu)

\noindent\phantom{\it E-mails }
\texttt{dcyang@bnu.edu.cn} (D. Yang)

\noindent\phantom {\it E-mails }
\texttt{wenyuan@bnu.edu.cn} (W. Yuan)


\begin{thebibliography}{99}

\bibitem{AFP2000}
L.~Ambrosio, N.~Fusco and D.~Pallara,
Functions of Bounded Variation and Free Discontinuity Problems,
Oxford Math. Monogr.,
The Clarendon Press, Oxford University Press, New York, 2000.
\vspace{-0.3cm}

\bibitem{BonamiPoornima1987}
A.~Bonami and S.~Poornima,
Nonmultipliers of the Sobolev spaces $W^{k,1}({\mathbb R}^n)$,
J. Funct. Anal. 71 (1987), 175--181.

\vspace{-0.3cm}

\bibitem{BMOPR2026}
S.~Bortz, K.~Moen, A.~Olivo, C.~P\'erez and E.~Rela,
Weighted Sobolev inequalities via the Meyers--Ziemer framework:
measures, isoperimetric inequalities, and endpoint estimates,
arXiv:2603.05382v1.

\vspace{-0.3cm}

\bibitem{BourgainBrezis2007}
J.~Bourgain and H.~Brezis,
New estimates for elliptic equations and Hodge type systems,
J. Eur. Math. Soc. (JEMS) 9 (2007), 277--315.

\vspace{-0.3cm}

\bibitem{BCS2026}
D.~Breit, A.~Cianchi and D.~Spector,
Sobolev inequalities for canceling operators,
J. Math. Pures Appl. (9) 207 (2026), Paper No. 103844, 34 pp.

\vspace{-0.3cm}

\bibitem{BCS2026jfaa}
D.~Breit, A.~Cianchi and D.~Spector,
Riesz potential estimates under co-canceling constraints,
J. Fourier Anal. Appl. 32 (2026), Paper No. 83, 36 pp.

\vspace{-0.3cm}

\bibitem{CoifmanRochberg1980}
R.~R.~Coifman and R.~Rochberg,
Another characterization of BMO,
Proc. Amer. Math. Soc. 79 (1980), 249--254.

\vspace{-0.3cm}

\bibitem{DPV2012}
E.~Di Nezza, G.~Palatucci and E.~Valdinoci,
Hitchhiker's guide to the fractional Sobolev spaces,
Bull. Sci. Math. 136 (2012), 521--573.

\vspace{-0.3cm}

\bibitem{EvansGariepy2015}
L.~C.~Evans and R.~F.~Gariepy,
Measure Theory and Fine Properties of Functions,
Revised edition,
CRC Press, Boca Raton, FL, 2015.

\vspace{-0.3cm}

\bibitem{Federer1969}
H.~Federer,
Geometric Measure Theory,
Grundlehren math. Wiss. 153,
Springer-Verlag New York, Inc., New York, 1969.

\vspace{-0.3cm}

\bibitem{FedererFleming1960}
H.~Federer and W.~H. Fleming,
Normal and integral currents,
Ann. of Math. (2) 72 (1960), 458--520.

\vspace{-0.3cm}

\bibitem{GHS2026arxiv}
J.~Goodman, F.~Hern\'andez and D.~Spector,
Potential estimates and Hodge systems with $L^1$ data on compact manifolds,
arXiv:2606.12891.

\vspace{-0.3cm}

\bibitem{GrafakosCFA}
L.~Grafakos,
Classical Fourier Analysis,
Third edition,
Grad. Texts in Math. 249,
Springer, New York, 2014.

\vspace{-0.3cm}

\bibitem{GrafakosMFA}
L.~Grafakos,
Modern Fourier Analysis,
Third edition,
Grad. Texts in Math. 250,
Springer, New York, 2014.

\vspace{-0.3cm}

\bibitem{LaceyMoenPerezTorres2010}
M.~T. Lacey, K.~Moen, C.~P\'erez and R.~H. Torres,
Sharp weighted bounds for fractional integral operators,
J. Funct. Anal. 259 (2010), 1073--1097.

\vspace{-0.3cm}

\bibitem{Lee2013}
J.~M.~Lee,
Introduction to Smooth Manifolds,
Second edition, Grad. Texts in Math. 218,
Springer, New York, 2013.

\vspace{-0.3cm}

\bibitem{LiuXiao2017}
L. Liu and J. Xiao,
Restricting Riesz--Morrey--Hardy potentials,
J. Differential Equations 262 (2017), 5468--5496.

\vspace{-0.3cm}

\bibitem{LiuXiao2021}
L. Liu and J. Xiao,
Fractional Hardy--Sobolev $L^1$-embedding per capacity-duality,
Appl. Comput. Harmon. Anal. 51 (2021), 17--55.

\vspace{-0.3cm}

\bibitem{LiuXiao2025}
L. Liu and J. Xiao,
Riesz-type calculus for Lorentz--Morrey spaces,
Adv. Math. 478 (2025), Paper No. 110414, 60 pp.

\vspace{-0.3cm}

\bibitem{Mazya2011}
V.~Maz'ya,
Sobolev Spaces with Applications to Elliptic Partial Differential
Equations,
Second, revised and augmented edition,
Grundlehren math. Wiss. 342,
Springer, Heidelberg, 2011.

\vspace{-0.3cm}

\bibitem{MeyersZiemer1977}
N.~G. Meyers and W.~P. Ziemer,
Integral inequalities of Poincar\'e and Wirtinger type for BV functions,
Amer. J. Math. 99 (1977), 1345--1360.

\vspace{-0.3cm}

\bibitem{MuckenhouptWheeden1974}
B.~Muckenhoupt and R.~L. Wheeden,
Weighted norm inequalities for fractional integrals,
Trans. Amer. Math. Soc. 192 (1974), 261--274.

\vspace{-0.3cm}

\bibitem{MTT02}
C.~Muscalu, T.~Tao and C.~Thiele,
Multi-linear operators given by singular multipliers,
J. Amer. Math. Soc. 15 (2002), 469--496.

\vspace{-0.3cm}

\bibitem{MyyrylainenPerezWeigt2024}
K.~Myyryl\"ainen, C.~P\'erez and J.~Weigt,
Weighted fractional Poincar\'e inequalities via isoperimetric inequalities,
Calc. Var. Partial Differential Equations 63 (2024), Paper No.~205, 32 pp.

\vspace{-0.3cm}

\bibitem{ONeil1963}
R.~O'Neil,
Convolution operators and $L(p,q)$ spaces,
Duke Math. J. 30 (1963), 129--142.

\vspace{-0.3cm}

\bibitem{Perez1995}
C.~P\'erez,
Sharp $L^p$-weighted Sobolev inequalities,
Ann. Inst. Fourier (Grenoble) 45 (1995), 1--16.

\vspace{-0.3cm}

\bibitem{PerezRela2019}
C.~P\'erez and E.~Rela,
Degenerate Poincar\'e--Sobolev inequalities,
Trans. Amer. Math. Soc. 372 (2019), 6087--6133.

\vspace{-0.3cm}

\bibitem{PerezRela2023}
C.~P\'erez and E.~Rela,
A tribute to Pola Harboure: isoperimetric inequalities and the HMS
extrapolation theorem,
Rev. Un. Mat. Argentina 66 (2023), 265--280.

\vspace{-0.3cm}

\bibitem{Ponce2016}
A.~C. Ponce,
Elliptic PDEs, Measures and Capacities: From the Poisson Equation
to Nonlinear Thomas--Fermi Problems,
EMS Tracts in Mathematics 23, European Mathematical Society,
Z\"urich, 2016.

\vspace{-0.3cm}

\bibitem{SSVS2017}
A.~Schikorra, D.~Spector and J.~Van Schaftingen,
An $L^1$-type estimate for Riesz potentials,
Rev. Mat. Iberoam. 33 (2017), 291--303.

\vspace{-0.3cm}

\bibitem{ShiehSpector2015}
T.-T.~Shieh and D.~E.~Spector,
On a new class of fractional partial differential equations,
Adv. Calc. Var. 8 (2015), 321--336.

\vspace{-0.3cm}

\bibitem{Spector2020}
D.~Spector,
An optimal Sobolev embedding for $L^1$,
J. Funct. Anal. 279 (2020), Paper No. 108559, 26 pp.

\vspace{-0.3cm}

\bibitem{Stein1970}
E.~M. Stein,
Singular Integrals and Differentiability Properties of Functions,
Princeton Mathematical Series 30, Princeton University Press,
Princeton, NJ, 1970.

\vspace{-0.3cm}

\bibitem{SteinWeiss1960}
E.~M. Stein and G.~Weiss,
On the theory of harmonic functions of several variables.
I. The theory of $H^p$-spaces,
Acta Math. 103 (1960), 25--62.

\vspace{-0.3cm}

\bibitem{VanSchaftingen2013}
J.~Van Schaftingen,
Limiting Sobolev inequalities for vector fields and canceling linear
differential operators,
J. Eur. Math. Soc. (JEMS) 15 (2013), 877--921.

\end{thebibliography}
\end{document}